\documentclass{base_class_arxiv}

\newcommand{\R}{\mathbb{R}}
\newcommand{\C}{\mathbb{C}}

\newcommand{\bi}{\mathbf i}

\DeclareMathOperator{\Pd}{Pd}

\usepackage{mathrsfs}
\newcommand{\VE}{\mathscr{E}}
\newcommand{\VF}{\mathscr{F}}
\newcommand{\VP}{\mathscr{P}}
\renewcommand{\pd}{\operatorname{Pd}}
\newcommand{\bC}{\mathbf{C}}
\newcommand{\bB}{\mathbf{B}}
\newcommand{\bA}{\mathbf{A}}

\renewcommand{\calK}{\mathcal{K}}
\newcommand{\scrL}{\mathscr{L}}
\newcommand{\scrH}{\mathscr{H}}
\newcommand{\dom}{\operatorname{dom}}

\title[Hard Lefschetz theorem and Hodge--Riemann relations]{Hard Lefschetz theorem and Hodge--Riemann relations for convex valuations}

\author[Andreas Bernig]{Andreas Bernig}
\author[Jan Kotrbat{\'y}]{Jan Kotrbat{\'y}}
\author[Thomas Wannerer]{Thomas Wannerer}

\address{Institut f\"ur Mathematik, Goethe-Universit\"at Frankfurt, Robert-Mayer-Str. 10, 60054 Frankfurt, Germany}
\email{bernig@math.uni-frankfurt.de}

\address{Charles University, Faculty of Mathematics and Physics, Sokolovsk\'a 49/83, 186 00 Prague, Czechia}
\email{kotrbaty@karlin.mff.cuni.cz}

\address{Friedrich-Schiller-Universit\"at Jena, Fakult\"at f\"ur Mathematik und Informatik, Institut f\"ur Mathematik, Ernst-Abbe-Platz 2, 07743 Jena, Germany}
\email{thomas.wannerer@uni-jena.de}

\thanks{AB was supported by DFG grant BE 2484/10-1.}
\thanks{JK was supported by  DFG grant BE 2484/5-2 and by Charles University grants PRIMUS/24/SCI/009 and UNCE/24/SCI/022.}
\thanks{TW was supported by DFG grant WA 3510/3-1.}


\begin{document}

\begin{abstract}
 The algebra of smooth translation-invariant  valuations on convex bodies, introduced  by S.~Alesker in the early 2000s, was in part  proved and in part conjectured to satisfy properties  formally analogous to those of the cohomology ring of a compact K\"ahler manifold:  Poincar\'e duality, the hard Lefschetz theorem, and the Hodge--Riemann relations.  Our main result establishes  the hard Lefschetz theorem and the Hodge--Riemann  relations in full generality. As a consequence, we obtain McMullen's quadratic inequalities, which are valid for strongly isomorphic simple polytopes and known to fail in general, for convex bodies with smooth and strictly positively curved boundary. Our proof is based on elliptic operator theory and on perturbation theory applied to unbounded operators on a natural Hilbert space completion of the space of smooth  translation-invariant valuations.
\end{abstract}

\maketitle

\section{Introduction}

\subsection{Geometric inequalities}

An extremely fruitful approach to the classical isoperimetric problem was discovered by Minkowski at the end of the nineteenth century. At its core lies the observation that the isoperimetric inequality can be formulated as an inequality between mixed volumes of convex bodies. This perspective has led to striking generalizations and a vast array of related results, yet it seems far from exhausted. In this respect, we prove here a collection of new inequalities for mixed volumes that in particular broadly generalizes the isoperimetric inequality.

The mixed volume is the polarization of the Lebesgue measure under the Minkowski addition on the space $\mathcal K(\RR^n)$ of convex compact sets in the $n$-dimensional euclidean space, i.e., the (unique) symmetric,  Minkowski-multilinear function $V:\mathcal K(\RR^n)^n \to \RR$ such that $V(A,\dots,A)=\vol(A)$. It is non-negative and  for all $A_1,\dots,A_n\in\mathcal K(\RR^n)$ it satisfies the fundamental Alexandrov--Fenchel inequality
\begin{align*}
V(A_1,A_2,A_3,\dots,A_n)^2\geq V(A_1,A_1,A_3,\dots,A_n)V(A_2,A_2,A_3,\dots,A_n).
\end{align*}

The Alexandrov--Fenchel inequality is of central importance in convex geometry and, in fact, far beyond. It subsumes many important inequalities for convex bodies, such as the Minkowski inequalities and the Brunn--Minkowski inequality. The two original proofs by Alexandrov \cite{Aleksandrov:Theorie2,Aleksandrov:Theorie4} and various subsequent alternative proofs  \cite{KavehKhovanskii,Cordero:OneMore,Gromov:ConvexSets, Wang:AF,Teissier:AF, BraendenLeake:LorentzianCones} revealed its deep, unexpected connections to different areas, in particular to algebraic and K\"ahler geometry. Let us also mention the recent breakthrough due to Shenfeld and van Handel  \cite{ShenfeldHandel:Extremals, ShenfeldHandel:Minkowski} in the open  problem of characterizing the equality cases (see also \cite{Schneider:AF,Schneider:AFzonoids,Schneider:BM, HugReichert:Support, HugReichert:Extremizers} for conjectures and partial results).  Last but not least, the Alexandrov--Fenchel inequality is intimately related to combinatorics  \cite{Stanley:TwoApplications, MaShenfeld:Extremals}; for a broader perspective on this connection see also \S\ref{ss:Hodge} below.

Alexandrov's first proof of the Alexandrov--Fenchel inequality was based on approximating convex bodies by strongly isomorphic simple polytopes. This approach was later extended in a seminal paper of McMullen \cite{McMullen:SimplePolytopes} where the following much more general result was proved (see also \cite{Timorin:Analogue,vanHandel:Shepard}):

\begin{theorem}[McMullen \cite{McMullen:SimplePolytopes}]
\label{thm:McMullen}
Let $0\leq k\leq \frac n2$, $N\in\NN$, $x_i\in\RR$, and let $A_j^i,C_l\in\mathcal K(\RR^n)$  be strongly isomorphic simple polytopes. If
\begin{align*}
\sum_{i=1}^N x_i V(A_1^i,\dots,A_k^i,C_0,\dots, C_{n-2k},\Cdot[k-1])=0,
\end{align*}
then
\begin{align*}
(-1)^k\sum_{i,j=1}^N x_ix_j V(A_1^i,\dots,A_k^i,A_1^j,\dots,A_k^j,C_1,\dots,C_{n-2k})\geq0
\end{align*}
with equality if and only if $
\sum_{i=1}^N x_i V(A_1^i,\dots,A_k^i,\Cdot[n-k])=0$.
\end{theorem}
The two equations in the statement of the theorem are to be understood as equations between functions on 
polytopes, $[k]$ stands for plugging in a polytope $k$ times. A mixed volume with more than $n$ arguments is zero by convention.

The case $k=0$ of Theorem \ref{thm:McMullen} is equivalent to the non-negativity of the mixed volume, which is a well-known, yet non-trivial fact. For $k=1$, the theorem yields the Alexandrov--Fenchel inequality. This is less obvious but still easy to see, cf. \cite{vanHandel:Shepard} or \cite[Corollary 8.1]{Kotrbaty:HR}. Since these two inequalities extend to \emph{all} convex bodies by continuity, it is tempting to expect the conclusion of Theorem \ref{thm:McMullen} to hold without any assumption on $A_j^i,C_l\in\mathcal K(\RR^n)$. However, a na\"ive approximation argument does not apply to $k\geq2$ and, in fact, a recent counterexample due to van~Handel \cite{vanHandel:Shepard} shows that the statement is \emph{wrong} for general convex bodies. A natural question therefore is whether McMullen's quadratic inequalities possibly hold for subclasses of convex bodies other than strongly isomorphic simple polytopes.

Inspired by Alexandrov's second proof of the Alexandrov--Fenchel inequality, we answer this question affirmatively by proving an analog  of Theorem \ref{thm:McMullen} for sufficiently smooth convex bodies. Namely, denoting the class of convex bodies with smooth and strictly positively curved boundary by $\mathcal K^\infty_+(\RR^n)$, we prove the following theorem.
\begin{MainTheorem}
\label{thm:mainMV}
Let $0\leq k\leq \frac n2$, $N\in\NN$, $x_i\in\RR$, and $A_j^i,C_l\in\mathcal K^\infty_+(\RR^n)$. If
\begin{align*}
\sum_{i=1}^N x_i V(A_1^i,\dots,A_k^i,C_0,\dots, C_{n-2k},\Cdot[k-1])=0,
\end{align*}
then
\begin{align*}
(-1)^k\sum_{i,j=1}^N x_ix_j V(A_1^i,\dots,A_k^i,A_1^j,\dots,A_k^j,C_1,\dots,C_{n-2k})\geq0
\end{align*}
with equality if and only if $\sum_{i=1}^N x_i V(A_1^i,\dots,A_k^i,\Cdot[n-k])=0$.
\end{MainTheorem}
The conclusion of Theorem \ref{thm:mainMV} in fact holds under the weaker assumption that the support functions of the bodies $A_j^i\in\calK(\RR^n)$ are of class $C^{1,1}$, see Corollary \ref{cor:K11}.

\subsection{Hodge theory for smooth valuations}

\label{ss:Hodge}

Although the statement of Theorem \ref{thm:McMullen} is purely convex geometric, there is a remarkable algebraic structure hidden in the background. In fact, the inequalities can be, and were originally, formulated in \cite{McMullen:SimplePolytopes} in terms of McMullen's polytope algebra, more precisely as the Hodge--Riemann relations for the subalgebra $\Pi(P)$ generated by polytopes strongly isomorphic to a given simple polytope $P$. For more details about the polytope algebra, the reader is referred to \cite{McMullen:PolytopeAlgebra, BernigFaifman:Polytopealgebra, Brion:Polytope}. Let us only mention here that the product in $\Pi(P)$ is induced by the Minkowski sum.

An analytic counterpart to McMullen's combinatorial theory was discovered in the 2000s through the seminal work of Alesker \cite{Alesker:Irreducibility, Alesker:Product}. Alesker's work along with many subsequent developments, in particular the paper \cite{BernigFu:Convolution} by the first-named author and Fu, revealed that the Minkowski sum, when restricted to convex bodies of class $\mathcal K^\infty_+(\RR^n)$, gives rise to an algebra $\Val^\infty$ of smooth valuations. The algebra $\Val^\infty$ exhibits numerous properties analogous to $\Pi(P)$ and even more structure, some of which is reviewed in \S\ref{ss:methods} below.  It satisfies Poincar\'e duality which turned out to have striking implications in integral geometry \cite{Fu:Unitary,BernigFu:Hig, BernigSolanes:Kinematic,KotrbatyWannerer:O2,Bernig:SU}. More recently, the investigation of the connection between the algebra of valuations and geometric inequalities was initiated \cite{Kotrbaty:HR, KotrbatyWannerer:MixedHR, KotrbatyWannerer:Harmonic,Alesker:Kotrbaty}.

In our paper, we introduce new techniques that allow us to pursue this study far beyond the previous partial results and complete the analogy between  $\Pi(P)$ and $\Val^\infty$. More precisely, we establish new properties of the algebra of smooth valuations that in particular imply Theorem \ref{thm:mainMV}.

To formulate our result, recall that a valuation is a function $\phi:\calK(\RR^n)\to\CC$ satisfying
\begin{displaymath}
\phi(K\cup L)=\phi(K)+\phi(L)-\phi(K\cap L)
\end{displaymath}
whenever $K,L,K\cup L\in\calK(\RR^n)$. The Banach space of translation-invariant, continuous valuations is denoted by $\Val$.  Alesker has discovered  a natural dense subspace $\Val^\infty\subset \Val$ of  smooth valuations.  We postpone the precise definition to \S\ref{sec:smooth_valuations}; let us only mention here that the function
\begin{equation*}
\phi(K)= V(A_1,\dots,A_k,K[n-k])
\end{equation*}
is a smooth valuation for any fixed $A_i \in \mathcal K^\infty_+(\R^n)$. In fact, it follows from a deep theorem of Alesker \cite{Alesker:Irreducibility} together with a recent observation made independently by Knoerr \cite{Knoerr:MV} and van Handel  that such valuations span  $\Val^\infty$.  Equivalently, the same holds for the valuations $\vol(\Cdot+A)$, $A\in\calK_+^\infty(\RR^n)$. The space of smooth valuations decomposes as
\begin{align*}
\Val^\infty=\bigoplus_{k=0}^n\Val_k^\infty,
\end{align*}
where $\Val_k^\infty\subset\Val^\infty$ is the subspace of $k$-homogeneous valuations. The convolution product defined by
\begin{displaymath}
\vol(\Cdot+A_1)*\vol(\Cdot+A_2)=\vol(\Cdot+A_1+A_2)
\end{displaymath}
is continuous in the natural Fr\'echet space topology  on $\Val^\infty$  and turns it into a commutative, associative, graded algebra with Poincar\'e duality.

For $C\in\mathcal K^\infty_+$ we define the map $L_C:\Val_k^\infty\to\Val_{k-1}^\infty$ by
\begin{displaymath}
 (L_C \phi) (K)=\left.\frac{d}{dt}\right|_{t=0} \phi(K+tC).
\end{displaymath}
Equivalently, $L_C$ is given by the convolution product with $n V(C,\Cdot[n-1])$. For any tuple $\bC=(C_1,\dots,C_m)$ of such convex bodies we set $L_\bC=L_{C_1}\circ\cdots\circ L_{C_m}$ which equals, up to a positive constant, the convolution with $V(C_1,\dots,C_m,\Cdot[n-m])$. Our main result is:

\begin{MainTheorem}
\label{thm:main}
Let $0\leq k\leq  \frac n2$. Fix a tuple $\bC_0=(C_0,\ldots,C_{n-2k})$ of convex bodies from $\mathcal K^\infty_+(\RR^n)$ and denote $\bC=(C_1,\ldots,C_{n-2k})$.
\begin{enuma}
\item \emph{Hard Lefschetz theorem.} The map $L_{\bC}:\Val_{n-k}^\infty\to\Val_k^\infty$ is an isomorphism of topological vector spaces.
\item \emph{Hodge--Riemann relations.} The sesquilinear form
\begin{displaymath}
q_{\bC_0}(\phi)=(-1)^k\b\phi*L_{\bC}\phi
\end{displaymath}
on $\ker L_{\bC_0}\subset\Val_{n-k}^\infty$ is positive definite.
\end{enuma}
\end{MainTheorem}

Theorem~\ref{thm:main} confirms a conjecture of the second-named author \cite[Conjecture D]{Kotrbaty:HR} and generalizes many important partial results, previously obtained in various works. In the special case when all the bodies $C_i$ are euclidean balls, the hard Lefschetz theorem was proved in \cite{Alesker:HLComplex, BernigBroecker:Rumin}  and the Hodge--Riemann relations in \cite{Kotrbaty:HR, KotrbatyWannerer:Harmonic}. For arbitrary bodies $C_i$, however, Theorem \ref{thm:main} was known only for $k\leq 1$ \cite{KotrbatyWannerer:MixedHR}. As observed by van Handel, the previous partial results implied in particular the full statement in dimensions $n\leq4$.

It is a remarkable phenomenon that versions of the hard Lefschetz theorem and Hodge--Riemann relations hold across a broad spectrum of separate theories. Indeed, apart from K\"ahler and algebraic geometry, from which  statements of this type acquired their names,  and McMullen's polytope algebra $\Pi(P)$ mentioned above, they have been proved in various different contexts on the frontier between geometry and combinatorics \cite{ADH:Lagrangian,AHK:Hodge,EliasWilliamson:Hodge,HuhWang:Enumeration,Karu:HL,BHMPW:Singular}. All these recent developments had striking consequences, in particular to combinatorial inequalities, many of which had been open for decades (see also the ICM addresses of Huh \cite{Huh:ICM18,Huh:ICM22}).

\subsection{Our methods}
\label{ss:methods}

There is a fundamental difference between Theorem~\ref{thm:main} and the analogous statement in any other context mentioned above. Namely, to the best of our knowledge, $\Val^\infty$ is the only \emph{infinite-dimensional} algebra where the hard Lefschetz theorem and Hodge--Riemann relations are known to hold. This fact makes the proof somewhat more delicate, in particular, one cannot use the general approach of Cattani \cite{Cattani:Mixed,Cattani:DescentLemma} and deduce the general statement directly from the previously established case of euclidean balls. Moreover, in  a finite-dimensional situation, the hard Lefschetz theorem and Hodge--Riemann relations are usually proven simultaneously, using an inductive argument. This general scheme again fails for $\Val^\infty$ since in particular the injectivity and surjectivity of the Lefschetz map $L_\bC$ are not equivalent. 

The infinite-dimensional setting of  Theorem \ref{thm:main} suggests that one should try to  reformulate the statement in  a way that is susceptible to the methods of functional analysis, even if it thus becomes formally stronger than the original problem. In this respect, our key observation is that there exist two natural Hilbert space completions  $\VE_k$ and $\VF_k$ of $\Val^\infty_k$. Their existence is less obvious in the language of mixed volumes, but naturally suggests itself when describing smooth valuations in terms of smooth differential forms, see \S\ref{sec:EF}.
	
We prove the following statement which directly implies the Hodge--Riemann relations (Theorem~\ref{thm:main}(b)) and thus Theorem~\ref{thm:mainMV}. 	
	
\begin{MainTheorem}\label{thm:mainE} 
 With the notation of Theorem \ref{thm:main}, the sesquilinear form $q_{\bC_0}$, densely defined on $\ker L_{\bC_0} \cap \Val^\infty_{n-k}\subset \ker(L_{\bC_0}\colon \VE_{n-k}\to \VE_{k-1})$, is closable. Its closure $q^\VE_{\bC_0}$ is the quadratic form of a unique self-adjoint operator $Q_{\bC_0}^\VE$ on $\ker L_{\bC_0}$.  The operator $Q_{\bC_0}^\VE$ is strictly positive, has discrete spectrum, and its eigenvalues have finite multiplicities. 
\end{MainTheorem}

The basic idea of our proof, which  goes back to Hilbert's proof of the Brunn--Minkowski inequality and  was later generalized by  Alexandrov in his  second proof of the Alexandrov--Fenchel inequality, is to continuously deform a given   tuple $\bC_0$ of smooth convex bodies into euclidean balls,  arguing that positivity of the sesquilinear form $q_{\bC_0}$  is preserved during this process. 
The perturbation theory of unbounded operators  on a Hilbert space, in particular the Kato--Rellich theorem, will provide us with the desired theoretical framework to bring this approach to fruition. 

Our strategy is to first establish the hard Lefschetz theorem for 
\begin{displaymath}
 L_\bC\colon \VE_{n-k}\to \VE_{k}
\end{displaymath}
and to use it to define a family of (unbounded) operators to which  perturbation theory can be applied.
Unlike Hilbert and Alexandrov we  cannot directly work with a natural elliptic differential operator associated to  $q_{\bC_0}$ and yet elliptic operator theory plays an important role in our proof. Elliptic  estimates are crucial in the proof of the surjectivity of $L_\bC\colon  \VE_{n-k}\to \VE_{k}$ and elliptic regularity allows us to obtain the surjectivity of $L_\bC$ also for smooth valuations.  An important ingredient, entering our argument at two steps, are the Hodge--Riemann relations in the context of complex linear algebra, a result due to Timorin \cite{Timorin:Mixed} and the most elementary incarnation of the Hodge--Riemann relations. We first use it to prove the injectivity of $L_\bC$ and later deduce from it the ellipticity of a particular second order differential operator.
	
Let us finally remark that in recent years  spectral estimates and the closely related $L^2$ method have played an important role in a number of high-profile results in convex geometry.  Let us mention  the proof of the (B) conjecture by  Cordero-Erausquin, Fradelizi, and Maurey \cite{EFM:BConjecture}, the spectral gap estimate of  Kolesnikov and Milman \cite{KolesnikovMilman:logBM,KolesnikovMilman:BL} in their work on  the log-Brunn--Minkowski inequality \cite{BLYZ:logBM}, the proof of the variance conjecture for unconditional convex bodies by Klartag \cite{Klartag:BerryEsseen},  and the confirmation of the  Gardner--Zvavitch conjecture \cite{KolesnikovLivshyts:GZ,EskenazisMoschidis:GZ}.

\subsection{Organization of the article}

In Section \ref{s:Background} we collect the necessary background from the theory of smooth valuations on convex bodies.  In Section \ref{s:SphereBundle} we develop the linear algebra of differential forms on the sphere bundle, connect it with mixed volumes and smooth valuations, and use it to prove  injectivity in the hard Lefschetz theorem. In Section \ref{s:Surjectivity}, after recalling the necessary results from elliptic operator theory, we  define the Hilbert space completions of the space of smooth valuations and prove surjectivity in the hard Lefschetz theorem. In Section \ref{s:HR} we use perturbation theory and previously established properties of the Lefschetz isomorphism to reduce Theorem \ref{thm:mainE} to the special case of euclidean balls. The latter is proved, along with further spectral properties of the Hodge--Riemann form, in Section \ref{s:Spectral} where we also relax the assumptions of Theorem~\ref{thm:mainMV}.

\subsection*{Acknowledgments}
We thank Semyon Alesker, Astrid Berg, Dmitry Faifman, Joe Fu, Ramon van Handel, and Tobias Weth for useful discussions. The first-named author thanks the ESI Vienna for hosting a research visit that contributed to the results of the present paper.

\section{Background and notation}
\label{s:Background}

In this section we collect for later use definitions and  results from the theory of valuations. For the purposes of this paper it suffices to work in  $\RR^n$ with its standard euclidean inner product and orientation. Let us emphasize however that all concepts we introduce,  including the convolution of valuations, are natural in the sense that they can be formulated in a way that does not depend on these choices. For a more complete presentation of the principal results of the subject we refer the reader to \cite[Chapter 6]{Schneider:BM} and \cite{Alesker:Kent,AleskerFu:Barcelona}.

\subsection{Translation-invariant continuous valuations}
Throughout the article, $B^n\subset\RR^n$ will denote the euclidean unit ball and $S^{n-1}=\partial B^n$ the euclidean unit sphere.

A convex body  is a non-empty convex compact subset of $\RR^n$.  The set of convex bodies, denoted by $\calK(\RR^n)$, is a locally compact metric space when equipped with the Hausdorff metric.

\begin{definition}
	A valuation is a function $\phi\colon \calK(\RR^n)\to \CC$ satisfying 
	\begin{displaymath}
	 \phi(K\cup L)= \phi(K)+ \phi(L)-\phi(K\cap L)
	\end{displaymath}
	for any $K,L\in\calK(\RR^n)$ whenever the union $K\cup L$ is convex.
\end{definition}

A valuation is called translation-invariant if $\phi(K+x)=\phi(K)$ holds for all $K\in\calK(\RR^n)$ and $x\in \RR^n$.
The vector space $\Val$ of translation-invariant, continuous valuations on $\calK(\RR^n)$ admits a natural grading by the degree of homogeneity. Namely, denoting by $\Val_k\subset\Val$ the subspace of valuations satisfying $\phi(tK)= t^k \phi(K)$ for $t>0$ and $K\in\calK(\RR^n)$, it is a fundamental result of McMullen \cite{McMullen:EulerType} that
\begin{equation}
\label{eq:McMullenGrading}
 \Val= \bigoplus_{k=0}^n\Val_{k}.
\end{equation}
An important  example of a  $k$-homogeneous, translation-invariant continuous valuation is given by the mixed volume
\begin{displaymath}
 \phi(K)= V(A_1,\ldots, A_{n-k},K[k]),
\end{displaymath}
where $A_1,\ldots, A_{n-k}$ are arbitrary fixed convex bodies.  The spaces $\Val_0$ and $\Val_n$ are one-dimensional; $\Val_0$ consists of the constant valuations, $\Val_n$  is spanned by the Lebesgue measure $\vol$. The latter is a non-trivial fact due to Hadwiger \cite{Hadwiger:Vorlesungen}.
The spaces $\Val_k$ for $0<k<n$ are infinite-dimensional.
The topology on $\Val$ is the topology of uniform convergence on compact subsets. As observed e.g.\ in \cite{Alesker:McMullen}, it is a consequence of \eqref{eq:McMullenGrading} that this topology is in fact induced by the Banach space norm
\begin{displaymath}
 \|\phi\| = \sup_{K\subset B^n} |\phi(K)|.
\end{displaymath}

\subsection{Smooth valuations}\label{sec:smooth_valuations}
The theory of valuations on convex bodies, as an established line of research, goes back to the solution of Hilbert's 3rd problem by Dehn. In the early 2000s, the work of Alesker completely reshaped this classical subject. At the center of this transformation is Alesker's discovery of the rich algebraic structure of the natural dense subspace of smooth valuations $\Val^\infty\subset \Val$. Below we present only a very narrow, not to say distorted, image of the modern theory of valuations, focusing only on what is strictly necessary for this paper. In particular, we  discuss neither the Alesker product  \cite{Alesker:Product}, the Fourier transform \cite{Alesker:Fourier,FW:Fourier}, nor other natural operations on valuations.

To define smooth valuations, we need to introduce more terminology. First, the  group $\GL(n)$ acts on the Banach space $\Val$ continuously by
\begin{displaymath}
 (g\cdot \phi)(K)= \phi(g^{-1}K).
\end{displaymath}
Second, the normal cycle of a convex body $K\in\calK(\RR^n)$ is, as a set, 
\begin{displaymath}
 \nc(K)=\{(x,u)\in S\RR^n\colon u \text{ outward unit normal at }x\in K\}.
\end{displaymath}
Here $S\RR^n =\RR^n \times S^{n-1}$ is the sphere bundle of $\RR^n$. In fact, $\nc(K)\subset S\RR^n$ is an oriented compact Lipschitz submanifold of dimension $n-1$ and thus defines a current through integration. For each $\omega\in \Omega^{n-1}(S\RR^n)$,
\begin{align}
\label{eq:omega}
 \phi(K) = \int_{\nc(K)} \omega
\end{align}
is a continuous valuation \cite{Alesker:VMfdsIII}. All differential forms we consider throughout this article are complex-valued. The cartesian product structure of the sphere bundle induces a bigrading $\Omega^k(S\RR^n) = \bigoplus_{i+j=k}\Omega^{i,j}(S\RR^n)$. The subspaces of forms invariant under the translations $(x,u)\mapsto (x+y,u)$ are denoted by $\Omega^k(S\RR^n)^{tr}$ and
 $\Omega^{i,j}(S\RR^n)^{tr}$. If $\omega\in\Omega^{k,n-k-1}(S\RR^n)^{tr}$, then the valuation \eqref{eq:omega} is in $\Val_k$. Finally, let $\calK_+^\infty(\RR^n)\subset \calK(\RR^n)$ be the class of convex bodies with $C^\infty$-smooth boundary and strictly positive Gauss curvature.

\begin{theorem}\label{thm:smooth}
For  $\phi\in\Val$ the following are equivalent:
\begin{enuma}
	\item $\phi$ is a smooth vector for the action of $\GL(n,\RR)$ on $\Val$, i.e., the map $\GL(n,\RR)\to \Val$, $g\mapsto g\cdot \phi$, is smooth.
	\item There exist  $\omega\in \Omega^{n-1}(S\RR^n)^{tr}$  and $c\in\CC$ such that
	\begin{equation} \label{eq:normal_cycle}
	  \phi(K)=c\vol(K) + \int_{\nc(K)} \omega,\quad K\in\calK(\RR^n).
	\end{equation}
	\item There exist $m\in\NN$, $k_j\in\{0,\dots,n\}$, $z_j\in \CC$, and $A^j_i\in \calK_+^\infty(\RR^n)$ such that
$$
	\phi(K) = \sum_{j=1}^m z_j V( A_1^j, \ldots, A_{n-k_j}^j,K[k_j]),\quad K\in\calK(\RR^n).
$$
\end{enuma}
\end{theorem}

\begin{definition} A valuation $\phi\in\Val$ is called \emph{smooth} if it satisfies one and hence all of the properties listed in Theorem~\ref{thm:smooth}.
\end{definition}

Alesker \cite{Alesker:Product} originally introduced smooth translation-invariant valuations via (a). By general representation theory, an immediate consequence of this definition is that the subspace $\Val^\infty\subset\Val$ of smooth valuations is dense. It is not difficult to see that the valuations satisfying (b) and (c) are smooth vectors in the sense of (a). The implication (a)$\implies$(b), by contrast, is highly non-trivial and follows from Alesker's irreducibility theorem \cite{Alesker:Irreducibility} and the Casselman--Wallach theorem \cite{Casselman:HCmod}. That (b) implies (c) was  observed independently by Knoerr~\cite{Knoerr:MV} and van Handel.

Both (a) and (b) can be used to define the topology on $\Val^\infty$. Namely, $\Val^\infty$ can be topologized either as a subspace of $C^\infty(\GL(n),\Val)$ or as a quotient of $\Omega^{n-1}(S\RR^{n})^{tr}\oplus \CC$; the resulting topologies coincide.

The valuation \eqref{eq:normal_cycle} will also be denoted by $[[c,\omega]]$. By the kernel theorem due to the first-named author and Br\"ocker \cite{BernigBroecker:Rumin},  $\phi=[[c,\omega]]=0$ if and only if 
\begin{displaymath}
  \phi(\{0\})=0 \quad \text{and}\quad D\omega+ c\cdot  dx_1\wedge\cdots \wedge  dx_n=0,
\end{displaymath}
 where $D$ is the Rumin differential  \cite{BernigBroecker:Rumin, Rumin:Contact} defined as follows.   The contact form on $S\RR^n$ is the $1$-form
 \begin{displaymath}
  \alpha|_{x,u} = \sum_{i=1}^n u_i dx_i.
 \end{displaymath}
A form $\eta\in \Omega(S\RR^n)$ is called vertical if $\alpha \wedge \eta=0$. If $\omega \in \Omega^{n-1}(S\RR^n)$, then there is a unique vertical form $\eta \in \Omega^{n-1}(S\RR^n)$ such that $d(\omega+\eta)$ is vertical and we set $D\omega=d(\omega+\eta)$.  We will also need the following fact.

 \begin{lemma}
 \label{lem:closed_vertical}
 Let  $1\leq k\leq n$ and  $\tau\in \Omega^{k,n-k}(S\RR^n)^{tr}$. There are $c\in \CC$ and  $\omega\in \Omega^{k,n-1-k}(S\RR^n)^{tr}$ such that
 \begin{equation}\label{eq:tauRumin} \tau= D\omega + c\, dx_1\wedge \cdots \wedge dx_n\end{equation}
 if and only if  
 \begin{equation}\label{eq:closed_vertical} d\tau=0, \quad  \alpha \wedge \tau=0,\quad \text{and}\quad \int_{S^{n-1}}\tau=0.
 \end{equation}\end{lemma}

\begin{proof} It follows directly from the definition of the Rumin differential that $\tau$ satisfies \eqref{eq:closed_vertical} if it has the form \eqref{eq:tauRumin}. Conversely, let $\tau$ satisfy \eqref{eq:closed_vertical}. If $k=n$, then $d\tau=0$ implies $\tau = c\, dx_1\wedge \cdots \wedge  dx_n$. For $1<k<n$ the cohomology group  $H^{n-k}(S^{n-1})$  vanishes and hence there is  $\omega\in \Omega^{k,n-k-1}(S\RR^n)^{tr}$ such that $\tau= d\omega$.  The same holds for $k=1$ since $\int_{S^{n-1}} \tau =0$ by assumption. Since $\tau$ is vertical, one has $d\omega=D\omega$.
\end{proof}

Note that by convention the integral condition in \eqref{eq:closed_vertical} is automatically satisfied for $k>1$. Lemma \ref{lem:closed_vertical} and the kernel theorem imply at once one more equivalent characterization of smooth translation-invariant valuations, cf. \cite[Section 8]{AleskerBernig:Product}.

\begin{proposition} 
The map
\begin{displaymath}
   \Val^\infty\to\{(\tau,z)\mid \tau\in \Omega^n(S\RR^n)^{tr}\text{ satisfies \eqref{eq:closed_vertical}}, z\in\CC\}
\end{displaymath}
given by
 \begin{displaymath}
 	\phi=\lcur c,\omega\rcur\mapsto \left(D\omega + c\,dx_1\wedge \cdots \wedge dx_n, \phi(\{0\})\right)
  \end{displaymath}
is a well-defined isomorphism.
\end{proposition}

\subsection{The convolution of smooth valuations}

The first-named author and Fu \cite{BernigFu:Convolution} have discovered a continuous bilinear map $\Val^\infty \times \Val^\infty \to  \Val^\infty$ called  convolution
that is uniquely determined by the following property:
\begin{equation}
	\label{eq:def_convolution} \vol(\Cdot + A_1) * \vol(\Cdot + A_2)= \vol (\Cdot + A_1+ A_2),\quad A_1,A_2\in \calK_+^\infty(\RR^n).
\end{equation}
The convolution gives $\Val^\infty$ the structure of a commutative graded algebra with $\vol$ as identity element. Explicitly, $\Val_{n-k}^\infty*\Val_{n-l}^\infty\subset\Val_{n-k-l}^\infty$. If  $\bA$ is a $k$-tuple and $\bC$ an $l$-tuple of convex bodies from $\calK_+^\infty(\RR^n)$ with $k+l\leq n$, then \eqref{eq:def_convolution} yields
\begin{align}
	V(\bA,\Cdot [n-k]) * V(\bC,\Cdot[n-l])  = \frac{(n-k)!(n-l)!}{n!(n-k-l)!}V(\bA,\bC,\Cdot[n-k-l]).
 \label{eq_convolution_mixed_volumes}
\end{align}
In particular, $V(A_1,\ldots,A_k,\Cdot[n-k])$ and $V(A_1,\Cdot[n-1])  * \cdots * V(A_k,\Cdot[n-1]) $ differ only by a positive factor.

The following description of the convolution in terms of differential forms will be crucial for our purposes.

Let $*\colon \largewedge^k (\RR^n)^* \to  \largewedge^{n-k} (\RR^n)^*$ be the Hodge star isomorphism given by $\alpha\wedge*\beta=\langle\alpha,\beta\rangle\vol$. We identify $\Omega^{k,l}(S\RR^n)^{tr}=\largewedge^k(\RR^n)^*\otimes\Omega^l(S^{n-1})$ and define the isomorphism $*_1: \Omega^{k,l}(S\RR^n)^{tr} \to \Omega^{n-k,l}(S\RR^n)^{tr}$ by
\begin{displaymath}
   *_1=(-1)^{\binom{n-k}{2}} \ * \otimes \,\mathrm{id}.
\end{displaymath}
Note that $	d \,*_1 =(-1)^n *_1 d$.

\begin{theorem}[{\cite{BernigFu:Convolution}}]\label{thm:BFconvolution}
Let $k_1+k_2 \leq n$ and suppose $\omega_i\in\Omega^{n-k_i,k_i-1}(S\RR^n)^{tr}$ for $i=1,2$ . Then $ \lcur 0, \omega_1\rcur * \lcur 0, \omega_2\rcur = \lcur 0, \omega\rcur$ with
\begin{displaymath}
 \omega= *_1 ^{-1}( *_1 \omega_1 \wedge *_1 D\omega_2).
\end{displaymath}
Moreover,
\begin{displaymath}
 *_1 D\omega= *_1 D\omega_1 \wedge *_1 D\omega_2.
\end{displaymath}
\end{theorem}

\subsection{Poincar\'e duality}

 The convolution of valuations defines a continuous bilinear form $\Val^\infty\times\Val^\infty\to\CC$, called the Poincar\'e pairing, by
\begin{equation}\label{eq:Poincare_pairing}\langle \phi, \psi \rangle =( \phi* \psi)(\{0\}). \end{equation}
In terms of differential forms, if $\phi_1=\lcur 0,\omega_1\rcur \in \Val_k^\infty$ and $\phi_2=\lcur 0,\omega_2\rcur \in \Val_{n-k}^\infty$ have complementary degrees  with $0<k<n$, then
\begin{displaymath}
\omega_1 \wedge D\omega_2\in \Omega^{n,n-1}(S\RR^n)^{tr}\simeq \Omega^{n-1}(S^{n-1})
\end{displaymath}
and 
\begin{equation} \label{eq_wannerer}
	\langle\phi_1,\phi_2\rangle=(-1)^k \int_{S^{n-1}} \omega_1 \wedge D\omega_2 ,
\end{equation}
see \cite[Proposition 4.2]{Wannerer:UnitaryAreaMeasures}. A key property of the product of valuations is the following

\begin{theorem}[\cite{Alesker:Product}]
The Poincar\'e pairing \eqref{eq:Poincare_pairing} is perfect, i.e., the induced map 
\begin{equation}\label{eq:smooth_generalized}
	\Val^\infty(\RR^n)  \to \Val^\infty(\RR^n)^* 
\end{equation}
is injective with dense image.
\end{theorem}

Elements of $\Val^{-\infty}=(\Val^\infty)^*$ are called generalized translation-invariant valuations, see \cite{AleskerFaifman:Lorentzgroup}. Observe that for $\phi\in \Val^{-\infty}$ and  $K\in\calK_+^\infty(\RR^n)$  it makes sense to define
\begin{equation}\label{eq:evaluation}
  \phi(K)= \langle \phi, \vol(\Cdot + K) \rangle,
  \end{equation}
since the equality is valid in the classical sense if we consider $\phi\in\Val^\infty$ as a generalized valuation via the embedding \eqref{eq:smooth_generalized}.

\section{Differential forms on the sphere bundle}
\label{s:SphereBundle}

The hard Lefschetz theorem and the Hodge--Riemann relations for the algebra $\Val^\infty$ will ultimately be reduced to the Hodge--Riemann relations in complex linear algebra as proven by Timorin. In this section we develop the linear algebra of differential forms on the sphere bundle necessary for this reduction and connect it with smooth translation-invariant valuations, in particular with mixed volumes. As a first application we will be able to prove injectivity in the hard Lefschetz theorem.

\subsection{Complex linear algebra}
\label{subsec:Timorin}

To begin with, we state Timorin's result, namely, the analog of Theorem \ref{thm:main} in the context of complex linear algebra.

Let $W$ be a real vector space of  dimension $2n$ with a complex structure $J\colon W\to W$.
We  denote by  $\largewedge_\CC^k W^*= \largewedge^k W^*\otimes_\RR \CC$ the space of complex-valued $k$-forms on the real vector space $W$.  Let
$\largewedge_\CC^1 W^* =  W' \oplus \b W '$ be the decomposition of the space of $\RR$-linear maps $W\to \CC$ into the $\CC$-linear and $\CC$-antilinear maps.
Then, denoting $\largewedge^{(p,q)}W^*=\largewedge^p W' \otimes \largewedge^q \b W'$, one has
\begin{equation} \label{eq:pqforms}
	\largewedge^k_\CC W^* = \bigoplus_{p+q=k} \largewedge^{(p,q)}W^*.
\end{equation}
Let us clarify that we deviate from the standard notation here since in our applications $W$ will usually be a direct sum of two vector spaces and $\largewedge^{p,q}W^*$ will refer to the bidegree decomposition in this other context.

A $(1,1)$-form $\omega \in \largewedge^{(1,1)}W^*$ is called positive if it is real, i.e., $\b \omega= \omega$, and  satisfies $\omega(w,Jw) >0$ for each non-zero  $w \in W$.

\begin{theorem}[Timorin \cite{Timorin:Mixed}] \label{thm_timorin}
Let $p+q \leq n$ and let $\omega_0,\omega_1,\ldots,\omega_{n-p-q}$ be positive $(1,1)$-forms. Set $\Omega=\omega_1 \wedge \ldots \wedge \omega_{n-p-q} \in \largewedge^{(n-p-q,n-p-q)}W^*$. Then the following holds:
	\begin{enuma}
		\item \emph{Hard Lefschetz theorem.} The wedge product with $\Omega$ is an isomorphism
		\begin{displaymath}
			\Omega:\largewedge^{(p,q)}W^* \to \largewedge^{(n-q,n-p)}W^*.
		\end{displaymath}
	\item  \emph{Hodge--Riemann relations.} Let
	\begin{displaymath}
		P^{(p,q)}=\{\omega \in \largewedge^{(p,q)}W^*: \omega_0 \wedge \Omega \wedge \omega=0\}
	\end{displaymath}
be the subspace of primitive forms. Then the sesquilinear form
	\begin{displaymath}
		q(\alpha,\beta)=\bi^{p-q} (-1)^{\binom{p+q}{2}} \alpha \wedge \b \beta \wedge \Omega
	\end{displaymath}
 on $P^{(p,q)}$ is positive definite.
	\end{enuma}
\end{theorem}

Here and in what follows we use the canonical orientation of $W$ to orient $\largewedge^{2n} W^*$ and the notation $\bi=\sqrt{-1}$.

\subsection{The linear algebra of $(p,p)$-forms}
We are going to apply the linear Hodge--Riemann relations in the situation $W=V\oplus V$, where $V$ is an $n$-dimensional real vector space and $J(v,w)=(-w,v)$.  To this end we first need to verify that a certain class of  forms is of  type $(p,p)$.

 Let $K^p(V)\subset \largewedge_\CC^p V^* \otimes \largewedge_\CC^p V^*$ be the kernel of  the map
 \begin{equation} \label{eq_contraction_map}
   \largewedge_\CC^p V^*\otimes \largewedge_\CC^p V^*\to \largewedge^{p-1}_\CC V^*\otimes \largewedge^{p+1}_\CC V^*
 \end{equation}
given by $ \beta\otimes \gamma \mapsto \sum_{i=1}^n \iota_{e_i}\beta \otimes \alpha_i\wedge \gamma,$
where $e_1,\ldots, e_n$ is a basis of $V$ and $\alpha_1,\ldots, \alpha_n$ is the dual basis. We need the following fact, proved for example in \cite[Lemma 5.7]{BFSW}.
\begin{lemma}
	$K^p(V)$ is irreducible under the natural action of $\GL(n,\CC)$. 
\end{lemma}

At the same time, we can view $K^p(V)$ as a subspace of  $\largewedge_\CC^{2p} W^*$, where $W= V\oplus V$. In this respect, the following fact will be important.

\begin{proposition}\label{prop:ppforms}
  \begin{displaymath}
    K^p(V)\subset \largewedge^{(p,p)} W^*.
  \end{displaymath}
\end{proposition}
\begin{proof}
	The decomposition \eqref{eq:pqforms} is compatible with the  natural action of $\GL(n,\CC)$, i.e., each summand $\largewedge^{(p,q)} W^*$ is an invariant subspace. 
	Observe that  $K^p(V)$ and $\largewedge^{(p,p)} W^*$ intersect non-trivially; indeed, if $(x_i,y_i)$ are coordinates on $W=V\oplus V$ and $dz_i=dx_i + \bi dy_i$, then
\begin{displaymath}
	dz_1\wedge \cdots\wedge dz_p \wedge d\b z_1\wedge \cdots\wedge d\b z_p=c \  dx_1\wedge \cdots dx_p\wedge  dy_1\wedge \cdots \wedge  dy_p,
\end{displaymath}
for some $c\neq 0$, lies in the intersection. Since $K^p(V)$ is irreducible, the claim follows.
\end{proof}

\subsection{Smooth valuations and $(p,p)$-forms}
Let us write $(x,u)$ for the standard coordinates on $\RR^n\times \RR^n$.
To each convex body $A\in \calK^\infty_+(\RR^n)$ we associate its support function
\begin{displaymath}
h_A(u)=\sup_{a\in A}\langle a,u\rangle,\quad u\in\RR^n,
\end{displaymath}
and the following smooth differential forms on $\RR^n\times (\RR^n\setminus\{0\})$:
\begin{align}\label{eq:omegaOmega}
	\begin{split}
	\alpha_A & = \sum_{i=1}^n  \pder[h_A]{u_i} dx_i,\\
	\omega_A & =  h_A \iota_{T} \vol, \\
	\Omega_A & = -d\alpha_A,
	\end{split}
\end{align}
where $T= \sum_{i=1}^n u_i \pder{x_i}$ denotes the Reeb vector field and $\vol$ is the volume form on the first factor.
By restriction we obtain differential forms on the sphere  bundle $S\RR^n=\RR^n\times S^{n-1}$ that we will not distinguish notationally.

\begin{lemma}\label{lemma:form_mixed_volume} Let $A\in \calK_+^\infty(\RR^n)$. We have
	\begin{displaymath}
	  n V(A,K[n-1])= \int_{\nc(K)}\omega_A, \quad K\in\calK(\RR^n),
	\end{displaymath}
	and 
$*_1 D\omega_A= \Omega_A$.
\end{lemma}
\begin{proof}
First, if $P\in\calK(\RR^n)$ is a polytope, then
\begin{displaymath}
  nV(P[n-1],A)= \sum_{F} h_A(\nu_F)\vol_{n-1}(F) =  \int_{\nc(P)}\omega_A,
\end{displaymath}
where the sum is over the facets of $P$ and $\nu_F$ denotes the unit facet normal. The general case of the first statement then follows by continuity.

To prove the second identity, we first show that  $ d(\omega_A+\alpha \wedge \xi)$ with
\begin{displaymath}
	\xi= -\sum_{i=1}^n \pder[h_A]{u_i}\iota_{\pder{x_i}} \iota_T \vol
\end{displaymath}
is vertical. We have
\begin{displaymath}
	d\omega_A=dh_A \wedge \iota_T \vol+h_A d \iota_T \vol,
\end{displaymath}
and the second term is vertical. Since
\begin{displaymath}
	d\alpha \wedge \xi = \sum_{k=1}^n du_k\wedge dx_k \wedge \xi =  - dh_A \wedge \iota_T\vol
\end{displaymath}
we obtain
\begin{displaymath}
	D\omega_A= d(\omega_A+\alpha \wedge \xi) = d\omega_A+d\alpha \wedge \xi-\alpha \wedge d\xi = h_A d \iota_T \vol-\alpha \wedge d\xi.
\end{displaymath}

Next, we get
\begin{align*}
	\alpha \wedge d\xi & = -\alpha \wedge \sum_{i,k=1}^n \frac{\partial^2 h_A}{\partial u_i \partial u_k} du_k \wedge \iota_{\frac{\partial}{\partial x_i}} \iota_T \vol+\alpha \wedge \sum_{i=1}^n \frac{\partial h_A}{\partial u_i} \iota_{\frac{\partial}{\partial x_i}} d\iota_T \vol.
\end{align*}
Since $\sum_{i=1}^n \pder[h_A]{u_i} u_i = h_A$ by homogeneity of $h_A$, the second term is equal to $h_A d\iota_T \vol$.

Using $\sum_{i=1}^n \frac{\partial^2 h_A}{\partial u_i \partial u_k}u_i=0$ in the third equality, the claimed identity now follows  from
\begin{align*}
	*_1D\omega_A & = *_1\left(\alpha \wedge \sum_{i,k=1}^n \frac{\partial^2 h_A}{\partial u_i \partial u_k} du_k \wedge \iota_{\frac{\partial}{\partial x_i}} \iota_T \vol\right) \\
	& = (-1)^n \sum_{i,k=1}^n \frac{\partial^2 h_A}{\partial u_i \partial u_k} *_1\left(\alpha \wedge \iota_{\frac{\partial}{\partial x_i}} \iota_T \vol \wedge  du_k\right)\\
	& = \sum_{i,k=1}^n \frac{\partial^2 h_A}{\partial u_i \partial u_k} dx_i \wedge du_k\\
	& = -d\alpha_A\\
	& =\Omega_A.
\end{align*}
\end{proof}

\begin{corollary}\label{cor:evaluation}
Let $A\in \calK_+^\infty(\RR^n)$. If $\phi= \lcur0,\omega\rcur$ for $\omega\in \Omega^{k,n-k-1}(S\RR^n)^{tr}$, then 
\begin{displaymath}
  \phi(A) = \frac{(-1)^{\binom{n}{2}}}{k!} \int_{S^{n-1}} \ast_1 \omega \wedge \Omega_A^k.
\end{displaymath}
\end{corollary}
\begin{proof}
Denote  $\psi_A=nV(A,\Cdot[n-1])$. By Lemma~\ref{lemma:form_mixed_volume} and Theorem~\ref{thm:BFconvolution},  we have
\begin{align*}
k!\phi(A)=(L_A^k\phi)(\{0\})=\langle\phi,\psi_A^k\rangle=(-1)^{\binom{n}{2}}\int_{S^{n-1}} \ast_1 \omega \wedge \Omega_A^k.
\end{align*}
\end{proof}

The  sphere bundle of any riemannian manifold has a natural contact structure given by a distribution of contact planes  $H_{x,u}\subset T_{x,u}SM$. In the special case $M=\RR^n$, each contact plane $H_{x,u}=  u^\perp \oplus u^\perp$ is a complex vector space  of dimension $n-1$ with the complex structure $J(X_1,X_2)= (-X_2,X_1)$.

We say  that a differential form $\omega\in \Omega^{p+q}(S\RR^n)$ is a $(p,q)$-form if each 
$ \omega|_{x,u}$ belongs to $\largewedge^{p+q} H^*_{x,u}$  and is a $(p,q)$-form in the sense of Subsection~\ref{subsec:Timorin}. Note that the first condition is equivalent to $\iota_T\omega=0$. A $(1,1)$-form $\omega\in \Omega^2(S\RR^n)$ is called positive if each $\omega|_{x,u}\in  \largewedge^{(1,1)} H^*_{x,u}$ is positive.

\begin{proposition} \label{prop_positivity_star1_domega}
If $\omega\in \Omega^{n-k,k-1}(S\RR^n)^{tr}$, then $*_1 D\omega$ is a $(k,k)$-form. Moreover, the $(1,1)$-form $*_1 D\omega_A$ is positive for each $A\in\calK_+^\infty(\RR^n)$.
\end{proposition}

\begin{proof}
 First, from $\alpha \wedge D\omega=0$ we get $\iota_T(*_1 D \omega)=0$, hence $*_1D\omega|_{x,u} \in \largewedge^{2k} H^*_{x,u}$ for all $(x,u)$. The equation $d\alpha \wedge D\omega=0$ implies that $*_1 D\omega|_{x,u}$ is in the kernel of the map \eqref{eq_contraction_map} (with $V=u^\perp$). Hence, by Proposition~\ref{prop:ppforms}, $*_1D\omega|_{x,u} \in K^k(u^\perp) \subset \largewedge^{(k,k)}H_{x,u}^*$.
	
	 Second, since $A$ has a strictly positively curved boundary,
\begin{displaymath}
  \Omega_A(X,JX)= \Hess h_A(u)(X_1,X_1) + \Hess h_A(u)(X_2,X_2)>0
\end{displaymath}
for every non-zero $X=(X_1,X_2)\in H_{x,u}$. 
\end{proof}

\subsection{Injectivity in the hard Lefschetz theorem}

We are now in a position to deduce  injectivity in the hard Lefschetz  theorem for smooth valuations from the linear Hodge--Riemann relations.

\begin{proposition} \label{prop_injectivity}
	For each tuple $\bC=(C_1,\ldots, C_{n-2k})$ of convex bodies from $\calK_+^\infty(\RR^n)$, the map $L_\bC\colon \Val^\infty_{n-k}\to \Val^\infty_{k}$ is injective.
\end{proposition}	

\begin{proof}

For $k=0$, this follows trivially from \eqref{eq_convolution_mixed_volumes} and $V(C_1,\ldots,C_n) \neq 0$. If $2k=n$, then $L_\bC$ is the identity and there is nothing to prove. Let us assume $0<2k<n$ from now on. By translation invariance, we may further assume that each $C_i$ contains the origin in its interior.

Let $\phi=\lcur0,\omega\rcur$ be in the kernel of $L_\bC$. Put $\tau=*_1 D\omega\in \Omega^{k,k}(S\RR^n)$. By \eqref{eq_convolution_mixed_volumes}, Theorem \ref{thm:BFconvolution}  and Lemma~\ref{lemma:form_mixed_volume},
\begin{equation}\label{eq:HLcondition}
	\Omega_{C_1} \wedge \cdots \wedge \Omega_{C_{n-2k}} \wedge \tau =0.
\end{equation}
 By Proposition~\ref{prop_positivity_star1_domega}, each $\tau|_{x,u}$ is a $(k,k)$-form and the above equation shows that it is primitive with respect to the positive $(1,1)$-forms \begin{displaymath}
		\Omega_{C_1}|_{x,u},\ldots,\Omega_{C_{n-2k}}|_{x,u}.
	\end{displaymath}

The linear Hodge--Riemann relations (Theorem \ref{thm_timorin}), for the complex vector space  $(H_{x,u}, J)$ of complex dimension $n-1$ and $p=q=k$ yield
\begin{displaymath}
 (-1)^k \tau\wedge \b \tau \wedge \Omega_{C_1}\wedge \cdots \wedge \Omega_{C_{n-2k-1}} \geq 0 \quad \text{at }(x,u)
\end{displaymath}
with equality if and only if  $\tau|_{x,u}=0$.  Each convex body $A\in\calK_+^\infty(\RR^n)$ with the origin in the interior satisfies
$\alpha_{A}|_{x,u}(u,0)= h_A(u)>0$. We thus have
\begin{equation} \label{eq:integrand}	
 	(-1)^{k}  \alpha_{C_{n-2k}} \wedge \tau\wedge \b \tau \wedge \Omega_{C_1}\wedge \cdots \wedge \Omega_{C_{n-2k-1}} \geq 0 \quad \text{at }(x,u)
\end{equation}
with equality if and only if  $\tau|_{x,u}=0$.

Since $\tau$ is closed and $H^k(S^{n-1})=0$, we may write $ \tau =d\kappa.$ Using Stokes' theorem and \eqref{eq:HLcondition},
\begin{align*} & \int_{S^{n-1}}  \alpha_{C_{n-2k}} \wedge \tau \wedge \b \tau \wedge \Omega_{C_1}\wedge \cdots \wedge \Omega_{C_{n-2k-1}}\\
	&\quad = -\int_{S^{n-1}}  \Omega_{C_{n-2k}} \wedge \tau \wedge \b \kappa \wedge \Omega_{C_1}\wedge \cdots \wedge \Omega_{C_{n-2k-1}}=0. 
\end{align*} 
This implies that equality holds in \eqref{eq:integrand} for every $(x,u)$  and thus $\tau=0$, which implies that $\phi=0$.
\end{proof}

\section{Surjectivity in the hard Lefschetz theorem}

\label{s:Surjectivity}

Let $\bC=(C_1,\ldots,C_{n-2k})$ be a tuple of convex bodies from $\calK_+^\infty(\RR^n)$. The goal of this section is to establish the surjectivity of $L_\bC\colon \Val^\infty_{n-k}\to \Val^\infty_k$, thus proving the hard Lefschetz theorem. As we have seen in the previous section, after a careful description of the Lefschetz map in terms of differential forms on the sphere bundle, the injectivity is an immediate consequence of Timorin's linear Hodge--Riemann relations and Stokes' theorem. Proving surjectivity, however, will require tools from PDE theory, in particular the  machinery of Sobolev spaces, weak solutions,  and elliptic regularity.

We will introduce in this section two natural Hilbert space completions, denoted $\VE_k$ and $\VF_k$, of the space of smooth valuations $\Val_{k}^\infty$. The existence of these two Hilbert space completions is a consequence of the two  descriptions of smooth  valuations via  smooth differential forms: forms of degree $n-1$ or closed vertical $n$-forms. In the present section, the  Hilbert space $\VE_k$ will play a  lesser role, but it will be of central importance for our proof of the Hodge--Riemann relations.

\subsection{Facts from elliptic operator theory}

Let us recall from \cite{LawsonMichelson:Spin} several fundamental results about elliptic operators. 

The Sobolev $s$ norm of a Schwartz function $u\colon \RR^n\to \CC^p$ is defined by 
\begin{equation}\label{eq:sobolev} 
	\|u\|_s = \left(\int_{\RR^n}  (1+|\xi|)^{2s}   |\hat u(\xi)|^2 d\xi\right)^\frac{1}{2},
\end{equation}
where $\hat u$ denotes the Fourier transform of $u$. Note that $\|\Cdot\|_0$ is the $L^2$ norm. 

Let $E\to X$ be a complex vector bundle over a compact manifold $X$. Let $U_i$, $i=1,\ldots, N$, be an open cover of $X$ by coordinate neighbourhoods such that $E$ is locally trivial over each $U_i$. Fix local coordinates, local trivializations, and a partition of unity $(\rho_i)_{i=1}^N$  subordinate to $(U_i)_{i=1}^N$. Let $\Gamma(E)$ denote the space of smooth sections of $E$. Any $u\in \Gamma(E)$ can be written as $u=\sum_{i=1}^N u_i$ with $u_i=\rho_i u$. For $s\in \RR$, the Sobolev $s$ norm on $\Gamma(E)$ is defined by 
\begin{displaymath}
   \|u\|_s^2 = \sum_{i=1}^N \| u_i\|_{s}^2,
\end{displaymath}
where, using $E|_{U_i} \cong U_i\times\CC^p$, the norms on the right-hand side are given by \eqref{eq:sobolev}. Different choices of the coordinate system or the local trivializations give equivalent Sobolev norms. The completion of $\Gamma(E)$ with respect to this norm is denoted $H^s(E)$. It is clear from \eqref{eq:sobolev} that $H^s(E)$ is  a Hilbert space.

Let $E$ and $F$ be complex vector bundles of ranks $p$ and $q$  over $X$. A linear map  $P\colon \Gamma(E)\to \Gamma(F)$ is called a  differential operator of order $m$ if each point of $X$ possesses an open neighbourhood $U$ with coordinates $x_1,\ldots,x_n$ and a local trivialization $E|_U\cong U\times \CC^p$ such that $P$ can be written as 
\begin{displaymath}
P = \sum_{|\alpha|\leq m}  A^\alpha \frac{\partial^{|\alpha|}}{\partial x^\alpha},
\end{displaymath}
where $A^\alpha$ are $q\times p$ matrices of complex-valued functions on $U$ and  $A^\alpha \neq 0$ for some $\alpha$ with $|\alpha|=m$. 

\begin{lemma} \label{lemma:sobolevExtension}
	Let $P\colon \Gamma(E)\to \Gamma(F)$ be a differential operator of order $m$ over a compact manifold $X$.  For any $s\in \RR$, the differential operator $P$ extends to a bounded operator $P\colon H^s(E)\to H^{s-m}(F)$.
\end{lemma}

The principal symbol of a differential operator $P$ is a section $\sigma(P)$ of $\pi^* \Hom(E,F)$, the pullback bundle of $\Hom(E,F)$ under the canonical projection $\pi\colon T^* X\to X$, such that for $\xi= \sum \xi_k  dx_k$,
\begin{displaymath}
  \sigma_\xi(P)=  \mathbf{i}^m \sum_{|\alpha|=m} A^\alpha(x) \xi^\alpha,
\end{displaymath}
where $\mathbf{i} =\sqrt{-1}$. 
If $P\colon \Gamma(E)\to \Gamma(F)$ and $Q\colon \Gamma(F)\to \Gamma(G)$ are differential operators, then $\sigma_\xi(Q\circ P)= \sigma_\xi(Q)\sigma_\xi(P)$. 
The differential operator $P$ is called elliptic if for non-zero cotangent vectors $\xi\in T^*X$ the symbol $\sigma_\xi(P)\colon E_x\to F_x$ is invertible for all $x\in X$. In particular $p=q$ in this case.

\begin{theorem}\label{thm:EOT1} Let $P\colon \Gamma(E)\to \Gamma(F)$ be an elliptic operator of order $m$ over a compact manifold $X$. 
	The following holds:
	\begin{enuma}
		\item \emph{Elliptic regularity.} For any open set $U\subset X$ and any $u\in H^s(E)$, 
		\begin{displaymath} 
			Pu|_U\in C^\infty \Longrightarrow  u|_U\in C^\infty.
	\end{displaymath}
		\item For each $s$, the kernel of $P\colon H^s(E)\to H^{s-m}(F)$ consists of smooth sections and its dimension is finite and independent of $s$. 
		\item \emph{Elliptic estimate.} For each $s$, there is  a constant $C$ such that 
		\begin{displaymath} 
			\|u\|_s \leq C  ( \|u\|_{s-m} + \|Pu \|_{s-m})
		\end{displaymath}
		for all $u\in H^s(E)$. 
		\item For each $s$, there is  a constant $C$ such that 
		\begin{displaymath} 
			\|u\|_s \leq C   \|Pu \|_{s-m}
	\end{displaymath}
		for all $u\in (\ker P)^\perp$, the  orthogonal complement of the kernel of $P$ in $H^s(E)$.

	\end{enuma}
	
\end{theorem}
\begin{proof}
	Items (a), (b),  and (c) are contained in \cite[Theorem 5.2]{LawsonMichelson:Spin}. A short argument (see, e.g., the proof of Lemma~4.14 in \cite{Wells:Complex}) shows that the existence of a parametrix implies  (d).	
\end{proof}

A hermitian metric on $E\to X$ and a smooth measure $\mu$ on $X$ define an $L^2$ norm
\begin{displaymath}
 \|u\|_{L^2} = \left(\int_X |u|^2 d\mu\right)^{1/2}
\end{displaymath}
on $\Gamma(E)$. If $P\colon \Gamma(E)\to \Gamma(F)$ is a differential operator, then the adjoint of $P$ is the differential operator $P^*\colon \Gamma(F)\to \Gamma(E)$ characterized by the property
\begin{displaymath}
 \langle Pu,v\rangle_{L^2} = \langle u ,P^* v\rangle_{L^2}
\end{displaymath}
for $u\in \Gamma(E)$ and $v\in \Gamma(F)$. The symbol of the adjoint satisfies $\sigma_\xi(P^*)= \sigma_\xi(P)^*$. 

\begin{example} Let $X$ be a compact riemannian manifold. 
	The Laplace--Beltrami operator  $\Delta\colon \Omega^k(X)\to \Omega^k(X)$ is defined by $\Delta= d \delta+ \delta d$, where we use the standard notation $\delta:=d^*$. It is an elliptic differential operator of order $2$ with principal symbol 
	\begin{displaymath}
	  \sigma_\xi(\Delta)= |\xi|^2.
	\end{displaymath}
\end{example}

\subsection{Hilbert space completions of the space of valuations}
\label{sec:EF}

Recall that if $Y\subset X$ is a closed subspace, then $X/Y$ is naturally a normed space with respect to $\|\pi(x)\|=\inf_{y\in Y}\|x+y\|$, $x\in X$, where $\pi:X\to X/Y$ is the canonical projection. 
Let   $i\colon X\to \wt X$ be the completion of $X$.
By the parallelogram law, if $X$ is an inner product space, then so are $X/Y$ and  $\wt X$. The following lemma is well known.

\begin{lemma}
	Let $X$ be an inner product space and $Y\subset X$ a closed subspace. Let $i\colon X\to \wt X$ be the completion of $X$. Then $\wt X/ \overline{i(Y)}\cong i(Y)^\perp\subset \wt X$ is the completion of $X/Y$.  
\end{lemma}

Let $E_{i,j}\to S^{n-1}$ denote the complex vector bundle $\largewedge_\CC^i (\RR^n)^*   \otimes \largewedge_\CC^j T^* S^{n-1}$ with the canonical hermitian metric. The space of smooth sections of $E_{i,j}$ is just $\Omega^{i,j}(S\R^n)^{tr}$.
Let us introduce the notation $\Omega_{H^s}^{i,j}(S\RR^n)^{tr}=H^s(E_{i,j})$ and $\Omega_{L^2}^{i,j}(S\RR^n)^{tr}=L^2(E_{i,j})$.
 Explicitly,  the $L^2$ inner product on  $\Omega_{L^2}^{i,j}(S\RR^n)^{tr}$  is given by
 \begin{equation} \label{eq_product_and_hodge}
 	\langle \omega_1,\omega_2\rangle_{L^2}=\int_{S^{n-1}} \omega_1 \wedge \overline{*\omega_2}, \quad \omega_1,\omega_2 \in \Omega^{i,j}(S\R^n)^{tr},
 \end{equation}
 where $*$ is the Hodge star operator. 

Define for $1 \leq k \leq n-1$ 
\begin{align*}
	X_k&= \Omega^{k,n-k-1}(S\RR^n)^{tr},\\
	Y_k&= \{\omega\in X_k \colon D\omega=0\},\\
	Z_k&= \{\tau \in \Omega^{k,n-k}(S\RR^n)^{tr}\colon \text{properties }  \eqref{eq:closed_vertical} \text{ hold}\}.
\end{align*}
	We consider  $X_k$ and $Z_k$ as  subspaces of $\Omega^{k,n-k-1}_{L^2}(S\RR^n)^{tr}$ and $\Omega^{k,n-k}_{L^2}(S\RR^n)^{tr}$ with the induced $L^2$ inner product.
To see that $Y_k$ is closed, just observe that the extension $D\colon \Omega^{k,n-1-k}_{L^2}(S\RR^n)^{tr}\to \Omega^{k,n-k}_{H^{-2}}(S\RR^n)^{tr}$ is continuous. 

\begin{definition}
	For $1 \leq k \leq n-1$, we denote by $\VE_k$ the Hilbert space completion of $X_k/Y_k$ and by $\VF_k$ the Hilbert space completion of $Z_k$, in both cases with respect to the induced $L^2$ inner product. Moreover, we set $\VE_0=\VF_0=\Val_0$ and $\VE_n=\VF_n=\Val_n$, with the norms defined in such a way that $\|1\|_\VE=\|1\|_\VF=1$ and $\|\vol\|_\VE=\|\vol\|_\VF=1$.
\end{definition}

\begin{example}
Let us describe these inner products  in the special case $k=n-1$ in more familiar terms. In this degree, every smooth valuation can be expressed as $\phi(K)=nV(A,K[n-1])-nV(C,K[n-1])$ with some convex bodies $A,C\in \calK_+^\infty$.  Hence by Lemma~\ref{lemma:form_mixed_volume}
$$ \|\phi\|_{\VE} = \| \omega_A-\omega_C\|_{X_{n-1}/Y_{n-1}}=  \inf_{x\in \RR^n} \left(  \int_{S^{n-1}} |h_{A+x}-h_C|^2 du\right)^{1/2}$$
 since $D(\omega_A-\omega_C)=0$ if and only if $h_A-h_C$ is linear. Similarly,
$$ \|\phi\|_{\VF} = \| D\omega_A-D\omega_C\|_{Z_{n-1}}=   \left(  \int_{S^{n-1}} \sum_{i,j=1}^n  \left|\frac{\partial^2 h_{A}}{\partial u_i\partial u_j}-\frac{\partial^2 h_{C}}{\partial u_i\partial u_j}\right|^2 du\right)^{1/2}.$$

\end{example}

Our first goal in this section is to prove a version of the Rellich--Kondrachov theorem. We will need the following lemma.

\begin{lemma}\label{lemma:tau_omega} 
Let $\tau\in \Omega^m_{L^2}(S^{n-1})$, $1\leq m \leq n-1$, satisfy $d\tau =0$ or, if $m=n-1$, let $\tau$ be perpendicular to  the riemannian volume form. Then there exists $\omega \in \Omega_{H^1}^{m-1} (S^{n-1})$ such that 
\begin{enuma}
\item $d\omega= \tau$ and 
\item $\omega\in \operatorname{im}(\delta\colon \Omega_{H^2}^m (S^{n-1})\to \Omega_{H^1}^{m-1} (S^{n-1}))$. 
\end{enuma}
Moreover, $\omega$ is uniquely determined by these two properties and there exists a constant $C=C(m,n)$ depending only $m$ and $n$ such that 
\begin{displaymath}
  \|\omega\|_{1} \leq C \| \tau\|_0.
\end{displaymath}
\end{lemma}

\begin{proof}
To prove existence observe that we can find $\sigma\in \Omega_{H^2}^m(S^{n-1})$ such that $\Delta \sigma  = \tau$. Indeed,   if $m<n-1$ then $\Delta$ is a bijection, and if $m=n-1$, the additional condition on $\tau$ implies that $\tau$ is orthogonal to the kernel of $\Delta$. By the definition of the Laplace--Beltrami operator it follows that  $\tau= d\omega + \delta\gamma$, where $\omega= \delta\sigma$ and $\gamma =d \sigma$. Since the exact and co-exact forms are orthogonal,
\begin{displaymath}
\| \delta\gamma\|_{L^2}^2 =\langle \delta\gamma, \delta\gamma\rangle_{L^2} = \langle \tau, \delta\gamma\rangle_{L^2} = \langle d\tau, \gamma\rangle_{L^2}=0.
\end{displaymath}
Thus $\tau= d\omega$ with $\omega\in \operatorname{im}(\delta\colon \Omega_{H^2}^m (S^{n-1})\to \Omega_{H^1}^{m-1} (S^{n-1}))$.

To prove uniqueness let $\omega=\delta \sigma$ and  suppose $d\omega=0$. Thus $\omega$ is harmonic and hence vanishes or, if $m=1$, is a constant function. But constant functions are perpendicular to the image of $\delta$ and hence $\omega=0$ in all cases.

Finally observe that $\delta\tau = \delta d \omega =\Delta \omega$. If $m>1$, then 
\begin{displaymath}
 \|\omega\|_1 \leq C \|\Delta \omega \|_{-1}\leq C\| \tau\|_0
\end{displaymath}
by Theorem~\ref{thm:EOT1}(d)  and Lemma~\ref{lemma:sobolevExtension}. If $m=1$, then $\omega$ is perpendicular to the kernel of $\Delta$, since $\omega$ is co-exact, and hence the inequalities are valid also in this case.
\end{proof}

\begin{proposition} \label{prop:SobolevVal}
	We have continuous  inclusions
	 \begin{displaymath}
	  \Val_k^\infty \subset \VF_k \subset \VE_k \subset \Val_k^{-\infty}.
	 \end{displaymath}
Moreover,  $\VF_k \hookrightarrow \VE_k$ is compact.
\end{proposition}

\begin{proof}
We may assume that $1 \leq k \leq n-1$. The $(n-1)$- and $n$-forms that make up the spaces $\VE_k$ and $\VF_k$ can be paired in the usual way with differential forms defining smooth valuations.  Thus $\VE_k, \VF_k \subset \Val_k^{-\infty}$ and \eqref{eq_wannerer} shows that $\Val^\infty_k \subset \VE_k, \VF_k$ in a way compatible with the former inclusions. 

Let $\tau\in Z_k$ and write $\tau= \sum_I dx_I\wedge \tau_I$ with $ \tau_I\in  \Omega^{n-k}(S^{n-1}) \subset  \Omega^{n-k}_{L^2}(S^{n-1})$, where $I$ runs over all multi-indices with $|I|= k$. Since  $\tau\in Z_k$, each $\tau_I$ satisfies the assumption of Lemma~\ref{lemma:tau_omega}. Thus we can find  a unique co-exact $\omega=\sum_I dx_I\wedge \omega_I\in\Omega_{H^1}^{k,n-1-k}(S\RR^n)^{tr}$ such that $d\omega= \tau$, and $\|\omega\|_{1} \leq C \| \tau\|_0$.  We thus get a map $Z_k \hookrightarrow \VE_k$ that can be extended by continuity to a map $\VF_k \hookrightarrow \VE_k$. Since $\Omega_{H^1}^{k,n-1-k}(S\RR^n)^{tr}\hookrightarrow \Omega_{L^2}^{k,n-1-k}(S\RR^n)^{tr}$ is compact by the Rellich--Kondrachov theorem  (see, e.g., \cite[Theorem 2.6]{LawsonMichelson:Spin}), it follows that the inclusion $\VF_k\hookrightarrow \VE_k$ is compact.
\end{proof}

\begin{remark}
The space $\Val^{-\infty}_k$ is natural in the sense that it is independent of the choice of a euclidean inner product and an orientation on $\RR^n$. The same is true for the subspaces  $\VE_k, \VF_k \subset \Val_k^{-\infty}$.
\end{remark}

The following theorem clarifies the connection between the two Hilbert space structures.

\begin{theorem}[Poincar\'e duality]
	\label{thm:poincare}
	The Poincar\'e pairing defined via convolution $\Val_k^\infty \times \Val_{n-k}^\infty\to \CC $,  extends to a continuous perfect pairing $\VE_k\times \VF_{n-k} \to \CC$. In particular, the induced map $\Pd_k\colon \VE_k \to (\VF_{n-k})^*$ is an isomorphism of Hilbert spaces and is called \emph{Poincar\'e duality}.
\end{theorem}

\begin{proof}
The Hodge star operator $*\colon \Omega^{k,l}(S\R^n)^{tr} \to \Omega^{n-k,n-1-l}(S\R^n)^{tr}$ is an isometry for the $L^2$ norms and satisfies $*^2=\mathrm{id}$. It is easy to check that for $1 \leq k \leq n-1$ it induces an isometric isomorphism
\begin{displaymath}
 *\colon Z_{n-k} \stackrel{\cong}{\longrightarrow} Y_k^\perp \cap X_k
\end{displaymath}
that extends to the $L^2$ completions. The statement of the theorem now follows easily from \eqref{eq_product_and_hodge} and \eqref{eq_wannerer}.
\end{proof}

\begin{lemma} \label{lemma_extension_convolution}
	If $\phi\in \Val^\infty_{n-k}$, then convolution with $\phi$ extends to continuous maps
	$\VE_l\to \VE_{l-k}$ and $\VF_l\to \VF_{l-k}$ for all $0\leq k \leq l \leq n$.
\end{lemma}

\begin{proof}
 The cases $k=0$ and $l=n$ are obvious and the case $k=l$ follows from Theorem \ref{thm:poincare}. If $0<k<l<n$, we have by Theorem \ref{thm:BFconvolution} well-defined maps $X_l/Y_l \to X_{l-k}/Y_{l-k}$ and $Z_l\to Z_{l-k}$ corresponding to the convolution with $\phi$ that are clearly continuous for the $L^2$ norm. 
\end{proof}

\subsection{Our elliptic operator}

Let $\bC=(C_1,\ldots, C_{n-2k})$ be a tuple of convex bodies from $\calK^\infty_+$, where $0 \leq k \leq \frac{n}{2}$. 
Recall that the elements of $\Omega^{i,j}(S\RR^n)^{tr}$ can be identified with the sections of the bundle $E_{i,j}$ over $S^{n-1}$. We define the differential operators 
\begin{align*}
	P_0 & :=d\colon \Gamma(E_{n-k,k})\to \Gamma(E_{n-k,k+1}),\\
	P_1 & :=  \alpha \colon  \Gamma(E_{n-k,k})\to \Gamma(E_{n-k+1,k}), \\
	P_2 & := d\alpha\colon \Gamma(E_{n-k,k})\to \Gamma(E_{n-k+1,k+1}), \\
	P_3 & := *_1^{-1}  (\Omega_\bC \wedge *_1\Cdot)\colon \Gamma(E_{n-k,k})\to \Gamma(E_{k,n-k}),
\end{align*}
where $\Omega_\bC=\Omega_{C_1} \wedge \cdots \wedge \Omega_{C_{n-2k}}$. Note that only $P_0$ is a differential operator of positive  order. The operators $P_1,P_2,P_3$ have  order $0$, i.e., are bundle morphisms. Note that $P_3=\mathrm{id}$ if $k=\frac{n}{2}$.
\begin{definition}
	Let  $P_\bC\colon \Gamma(E_{n-k,k})\to \Gamma(E_{n-k,k})$ be the differential operator of order two defined by
	\begin{displaymath}
	 P_\bC= P_0^*P_0 + \sum_{i=1}^3 P_i^*\Delta P_i.
	\end{displaymath}
	The adjoints are formed with respect to the canonical hermitian metric on $E_{n-k,k}$ and the canonical riemannian metric on $S^{n-1}$.
\end{definition} 

\begin{proposition}
	$P_\bC$ is  elliptic and injective for all $0<  k < \frac{n}{2}$.
\end{proposition}
\begin{proof}
	The symbols are given for any $\xi\in T^*S^{n-1}$ by
	\begin{align*} \sigma_\xi(P_0)(\tau) &= \mathbf i\xi \wedge \tau,\\
		\sigma_\xi(P_1)(\tau) &= \alpha \wedge \tau,\\
		\sigma_\xi(P_2)(\tau) &= d\alpha \wedge \tau,\\
		\sigma_\xi(P_3)(\tau) &= *_1^{-1}  (\Omega_\bC \wedge *_1\tau),
	\end{align*}
	and 
	\begin{displaymath}
	\sigma_\xi(P_\bC) = \sigma_\xi(P_0)^*\sigma_\xi(P_0) + |\xi|^2\sum_{i=1}^3  \sigma_\xi(P_i)^*\sigma_\xi(P_i). 
	\end{displaymath}
	Thus for $\xi\neq 0$ we have  $\sigma_\xi (P_\bC) (\tau)=0$ if and only if $\sigma_\xi(P_i)(\tau)=0$ for each  $i$.
The second and third equations, $\alpha\wedge \tau =0$ and $d\alpha \wedge \tau=0$, imply by Proposition~\ref{prop:ppforms} that 
	$ *_1 \tau$ is a $(k,k)$-form; the fourth, $*_1^{-1}  (\Omega_\bC \wedge *_1\tau)=0$, says that $*_1\tau$ is primitive with respect to the positive $(1,1)$-forms $\Omega_{C_1},\ldots,\Omega_{C_{n-2k}}$. Hence Timorin's theorem (Theorem \ref{thm_timorin}) yields
	\begin{equation}\label{eq:timorinElliptic}
	  (-1)^k {*_1\tau} \wedge \b{*_1\tau} \wedge \Omega_{\bC_{\setminus 1}} \geq 0,\end{equation}
   where $\bC_{\setminus 1} = (C_2,\ldots, C_{n-2k})$,
	with equality if and only if $*_1\tau=0$. But the first equation implies that $*_1\tau$ and $\b{*_1 \tau}$ are both divisible by $\xi$, which forces equality in \eqref{eq:timorinElliptic}. Thus $*_1\tau=0$, as claimed. 
	
	To prove that $P_\bC$ is injective, first note that for each $\tau\in \Gamma(E_{n-k,k})$,
	\begin{displaymath}
		\langle P_\bC\tau,\tau\rangle_{L^2} =\langle P_0\tau,P_0\tau\rangle+\sum_{i=1}^3\langle \Delta P_i\tau,P_i\tau\rangle_{L^2}\geq0
	\end{displaymath}
	is a sum of non-negative terms. Therefore, if $P_\bC\tau=0$, then $P_0\tau=0$ and $\Delta P_i\tau=0$ for $i=1,2,3$. 
 
	Recall that $P_1\tau \in \Gamma(E_{n-k+1,k})$ and  $P_3\tau \in \Gamma(E_{k,n-k})$.  Since $1 \leq k <n-1$, $\Delta$ is injective on $\Omega^k(S^{n-1})$. It follows that $P_1\tau=0$. We thus have $d\tau=0$ and $\alpha \wedge \tau=0$. In particular, there is $\phi\in\Val_{n-k}^\infty$ such that $\tau=\tau_\phi$.
	
	If $k>1$, $\Delta$ is injective on $\Omega^{n-k}(S^{n-1})$, hence $P_3\tau=0$. If $k=1$, we write $P_3\tau=\sum_{i=1}^n f_i(u) dx_i \wedge \vol_{S^{n-1}}$ with smooth functions $f_i \in C^\infty(S^{n-1})$, where $\vol_{S^{n-1}} \in \Omega^{n-1}(S^{n-1})$ is the volume form. Then on the one hand, $\alpha \wedge P_3\tau=0$ implies that $u_j f_i(u)=u_i f_j(u)$ for all $1 \leq i,j \leq n$, and on the other hand $\Delta P_3\tau=0$ implies that $f_i$ is constant for all $i$. This is only possible if all $f_i$ vanish, i.e., if $P_3\tau=0$. In both cases we thus get $L_\bC\phi=0$. By injectivity in the hard Lefschetz theorem (Proposition~\ref{prop_injectivity}), $\phi=0$ and hence $\tau=0$.
 \end{proof}

\subsection{Proof of surjectivity in the hard Lefschetz theorem}

Our strategy will be to first prove the existence of a solution $\psi$ in $\VE_{n-k}$ to the equation $L_\bC\psi =\phi$, where $\phi \in \Val_k^\infty$ and $0 < k < \frac{n}{2}$. Using our elliptic operator $P_\bC$ we will then be able to conclude that $\psi$ must be in fact smooth. The existence of a weak solution $\psi\in\VE_{n-k}$ rests on an $L^2$ estimate which is also a consequence of the ellipticity of $P_\bC$. 

For $0<k<n$ let  $\tau\colon \VE_{n-k}\to \Omega_{H^{-2}}^{n-k,k}(S\RR^n)^{tr}$ denote the continuous extension of $\lcur 0,\omega\rcur \mapsto D\omega$.
\begin{lemma} \label{lemma:P_on_tau}
 Let $\psi \in \VE_{n-k}$ with $0 < k < \frac{n}{2}$. Then
\begin{displaymath}
	P_\bC \tau_{\psi} = P_3^* \Delta  \tau_{L_\bC\psi}.
\end{displaymath}
\end{lemma}

\begin{proof}
Note that both sides are well defined  forms in $H^{-4}(E_{n-k,k})$ by Lemma \ref{lemma:sobolevExtension}. Suppose first that $\psi \in \Val^\infty_{n-k}$. Then $\tau_{L_\bC\psi}= P_3 \tau_\psi$, while $P_i\tau_{\psi} =0$ for $i=0,1,2$, which implies the displayed equation. The general case follows by approximation, using Lemma \ref{lemma:sobolevExtension} and Lemma~\ref{lemma_extension_convolution}.
\end{proof}

\begin{proposition}\label{prop:ellipticEst}
For $0 \leq k \leq \frac{n}{2}$, there exists a constant $c>0$, depending on $k$ and $\bC$, such that
	\begin{displaymath}
			\|\psi \|_{\VF}\leq c \|L_{\mathbf C}\psi\|_{\VF}
	\end{displaymath}
for all $\psi\in \VF_{n-k}$.
\end{proposition}

\begin{proof}
	The statement is trivial for $k=0$ (with $c=\frac{1}{n! V(C_1,\ldots,C_n)}$) and $k=\frac{n}{2}$ (with $c=1$). Let $0<k<\frac{n}{2}$. Let $\psi \in \Val^\infty_{n-k}$. Since $P_\bC$ is injective, Theorem~\ref{thm:EOT1}(d) and Lemma~\ref{lemma:sobolevExtension} imply that
 \begin{displaymath}
  \|\tau_{\psi}\|_{s} \leq c\|P_\bC\tau_{\psi}\|_{s-2} = c\|P_3^*\Delta\tau_{L_\bC\psi}\|_{s-2}
 	\leq c\|\Delta\tau_{L_\bC\psi}\|_{s-2} \leq  c\|\tau_{L_\bC\psi}\|_{s},
 \end{displaymath}
where in each step $c$ denotes a possibly different  constant. Choosing $s=0$ yields the desired inequality in the smooth case. The general case follows by approximation, using Lemma \ref{lemma_extension_convolution}.
\end{proof}

We are now ready to complete the proof of the hard Lefschetz theorem.

\begin{theorem}
	\label{prop_surjectivity}
For any $\phi \in \Val_{k}^\infty, 0 \leq  k  \leq \frac{n}{2}$ there is $\psi \in \Val_{n-k}^{\infty}$ such that
$L_\bC \psi=\phi$. 
\end{theorem}

\begin{proof}
	 The statement follows from injectivity of $L_\bC$ (Proposition \ref{prop_injectivity}) for $k=0$ and is trivial for $k=\frac{n}{2}$. We may therefore assume that $0<k<\frac{n}{2}$. By Proposition~\ref{prop:SobolevVal} we may consider $\phi$ as an element of $\VE_k$. Define the subspace
	\begin{displaymath}
		H:=\{L_\bC\rho \mid \rho\in\Val_{n-k}^\infty\}  \subset \VF_{k}
	\end{displaymath}
	and the following linear functional on $H$:
	\begin{displaymath}
		\ell(L_\bC\rho)=\langle\phi, \rho\rangle,
	\end{displaymath}
where $\langle\Cdot, \Cdot \rangle$ denotes the (extended) Poincar\'e pairing on $\VE_k \times \VF_{n-k}$. 
Note that $\ell$ is well defined by the injectivity part of the hard Lefschetz theorem, Proposition \ref{prop_injectivity}. Using the $L^2$ estimate of Proposition~\ref{prop:ellipticEst}, we obtain
\begin{displaymath}
 |\ell(L_\bC\rho)| \leq \|\phi\|_\VE \|\rho\|_\VF \leq c   \|\phi\|_\VE \|L_\bC \rho\|_\VF,
\end{displaymath}
hence $\ell$ is a bounded linear functional. 

By the Hahn--Banach theorem $\ell$ extends to a continuous linear functional $\tilde \ell$ on $\VF_k$. By Poincar\'e duality (Theorem~\ref{thm:poincare}), there is $\psi\in \VE_{n-k}$ such that $\tilde\ell = \langle  \psi, \Cdot\rangle$ and hence
\begin{displaymath}
   \langle\phi, \rho\rangle= \ell(L_\bC \rho)= \langle \psi, L_\bC \rho\rangle = \langle  L_\bC\psi, \rho\rangle
\end{displaymath}
for all $\rho\in\Val_{n-k}^\infty$. Consequently, we have  $L_\bC\psi =\phi$ and therefore  $P_\bC \tau_\psi=P_3^*\Delta \tau_\phi$ by Lemma \ref{lemma:P_on_tau}. Since the right-hand side is smooth if $\phi$ is smooth and $P_\bC$ is elliptic, Theorem~\ref{thm:EOT1}(a) implies that $\tau_\psi$ is smooth and hence $\psi \in \Val_{n-k}^{\infty}$.
\end{proof}

\begin{proof}[Proof of Theorem~\ref{thm:main}(a)] 
Proposition~\ref{prop_injectivity} and  Theorem~\ref{prop_surjectivity} imply that the map $L_\bC\colon \Val_{n-k}^\infty\to \Val_{k}^\infty$ is a bijection. Since $L_\bC$ is continuous, the open mapping theorem for Fr\'echet spaces implies that $L_\bC$ is an isomorphism of topological vector spaces. 
\end{proof}

\subsection{Continuous dependence on the reference bodies}

Identifying a convex body $A$ with its support function $ h_A$, we can view $\calK^\infty_+(\RR^n)$ as a subset of $C^2(S^{n-1})$. The topology inherited from this inclusion is called the $C^2$ topology on $\calK^\infty_+(\RR^n)$.  If $X$ and $Y$ are Banach spaces, we denote by $\scrL(X,Y)$ the Banach space of bounded linear operators from $X$ to $Y$.

The reason why we work  with the $C^2$ topology on $\calK^\infty_+(\RR^n)$ in the following proposition, rather than with $\calK^2_+(\RR^n)$, the class of convex bodies with  support function in $C^2$ and strictly positive radii of curvature, is that we do not know whether  the Lefschetz map $L_\bC$ is surjective for  reference bodies belonging merely to  $\calK^2_+(\RR^n)$. 

\begin{proposition}
\label{prop:contDependence}
Let $0 \leq k \leq \frac{n}{2}$. The  isomorphism $L_\bC\colon \Val^\infty_{n-k}\to \Val^\infty_k$ extends to an isomorphism $L_\bC
\colon \VF_{n-k}\to \VF_{k}$ of Hilbert spaces. Moreover, the maps
 $\calK^\infty_+(\RR^n)^{n-2k}\to \scrL(\VF_{n-k}, \VF_k)$,
\begin{displaymath}
	(C_1,\ldots, C_{n-2k})\mapsto L_{\bC},
\end{displaymath}
and $\calK^\infty_+(\RR^n)^{n-2k}\to \scrL(\VF_{k}, \VF_{n-k})$,
\begin{displaymath}
(C_1,\ldots, C_{n-2k})\mapsto (L_{\bC})^{-1}	,
\end{displaymath}
where the domain is equipped with the $C^2$ topology, are continuous.

\end{proposition}
\begin{proof}
	By Lemma \ref{lemma_extension_convolution} the map $L_\bC:\Val_{n-k}^\infty \to \Val_k^\infty$ extends to a continuous map $L_\bC:\VF_{n-k} \to \VF_k$. The elliptic estimate of Proposition~\ref{prop:ellipticEst} implies that the inverse $	L^{-1}_\bC \colon \Val_k^\infty \to \Val_{n-k}^\infty$
extends continuously to $\VF_k\to \VF_{n-k}$. This proves the first part. 
	
	The continuous dependence of $L_\bC$ on $\bC$ is clear from the definition of the norm on $\VF$. The continuity of $\bC\mapsto (L_\bC)^{-1}$ is an immediate consequence of the elementary fact that inversion is continuous in the operator norm.
\end{proof}

Our next goal is to show that the same conclusion holds if  $\VF$ is replaced by $\VE$.
We will show this using Poincar\'e duality.

\begin{lemma} \label{prop:Lefschetz_isomE}
		Let $C \in \calK^\infty_+(\R^n)$. Then the following diagram commutes:
	\begin{center}
		\begin{tikzpicture}
			\matrix (m) [matrix of math nodes,row sep=3em,column sep=4em,minimum width=2em]
			{ \VE_{k+1}  & (\VF_{n-k-1})^*\\
				\VE_{k} &  (\VF_{n-k})^*\\};
			\path[-stealth]
			(m-1-1) edge node [left] {$L^{\VE}_C$}  (m-2-1)
			(m-1-1) edge node [above] {$\pd_{k+1}$}  (m-1-2)
			(m-1-2) edge node [right] {$(L^\VF_C)^*$}  (m-2-2)
			(m-2-1) edge node [above] {$\pd_k$}  (m-2-2);
		\end{tikzpicture}
	\end{center}
\end{lemma}

\begin{proof}
	The identity $\langle L_C \phi,\psi\rangle = \langle  \phi,L_C\psi\rangle$  clearly holds for smooth valuations and extends by continuity to $\phi\in \VE_{k+1}$ and $\psi\in \VF_{n-k}$.
\end{proof}

\begin{corollary} \label{cor_contDependence_E}
Proposition~\ref{prop:contDependence} holds also with $\VF$ replaced by $\VE$. 
\end{corollary}

\begin{proof}
Lemma~\ref{prop:Lefschetz_isomE} immediately implies $\Pd_k \circ L_\bC^\VE =  (L_\bC^\VF)^* \circ \Pd_{n-k}$ or,  equivalently, $L_\bC^\VE = (\Pd_k)^{-1} \circ  (L_\bC^\VF)^* \circ \Pd_{n-k}$. Since taking duals and composition are continuous in the operator norm, the corollary follows from Proposition~\ref{prop:contDependence}.
\end{proof}

\section{Hodge--Riemann relations}

\label{s:HR}

In this section we use the properties of the Lefschetz isomorphism, in particular its continuous dependence  on the reference bodies, to prove the Hodge--Riemann relations. As mentioned already in the introduction, the basic idea is to continuously deform a given tuple of reference bodies $\bC=(C_0,\ldots, C_{n-2k})$ into euclidean balls $\bB=(B^n,\ldots, B^n)$  and to argue that positivity of the Hodge--Riemann form $q_\bC$  is preserved during this process.

The Hodge--Riemann form $q_\bC$  extends to a bounded quadratic form on $\VP_\bC\subset \VF_{n-k}$, the closed subspace of primitive valuations with respect to $\bC$. However,  $0$ will always belong to the spectrum of $q_\bC$ and therefore we cannot directly conclude that small perturbations of $q_\bC$ will remain positive.  Luckily we can also work with the Hilbert space $\VE_{n-k}$, in which  $q_\bC$ is strictly positive. The price we have to pay for this, however, is that our quadratic form becomes unbounded, which adds another layer of technical difficulties.

In this and the following section, we will slightly deviate from the previous notation. Since the additional body $C_0$ is important in the definition of the primitive subspaces and hence in the Hodge--Riemann form, we will denote by $\bC=(C_0,\ldots, C_{n-2k})$ the complete tuple of reference bodies and write $\bC_{\setminus0}:=(C_1,\ldots,C_{n-2k})$ if we wish to exclude $C_0$.

\subsection{Facts from perturbation theory}
For the convenience of the reader, we  review here some of the notions appearing in the statement of the Kato--Rellich theorem and refer  to the books \cite{ReedSimon:I,ReedSimon:II,Kato} for more information.

An operator in a Hilbert space $\scrH$ is a linear map $T\colon \dom(T)\to \scrH$ defined on a subspace $\dom (T)$ of $\scrH$. If $\dom(T)$, the  domain of $T$, is a dense subspace  of $\scrH$, then $T$ is said to be densely defined. Note that if $T$ is densely defined and  bounded, then $T$ can be uniquely extended to all of $\scrH$. 

An operator $T$ is called closed if its graph is a closed subset of $\scrH\times \scrH$. If the closure of the graph of $T$ happens to be the graph of an operator $\b T$, then $T$ is called closable. In this case $\b T$ is the smallest closed extension of $T$.

The adjoint $T^*$ of a densely defined operator  $T$  on $\scrH$ is defined as follows. Let $\dom(T^*)$ denote the set of all $x\in \scrH$ for which there is a $y\in\scrH$ with $ \langle Tz, x\rangle = \langle z, y\rangle$ for all $z\in \dom(T)$. For each such $x\in \dom(T^*)$ one defines $T^* x = y$. The adjoint need not  be densely defined and even $\dom(T^*)=\{0\}$ is possible.  A densely defined  operator $T$ is called self-adjoint if $\dom(T)=\dom(T^*)$ and $T^*=T$; 
 $T$ is called symmetric if $\dom(T) \subset \dom(T^*)$ and $T^*|_{\dom(T)}=T$. Another way to express this is to say that 
 \begin{displaymath}
   \langle Tx,y\rangle = \langle x,Ty\rangle
 \end{displaymath}
 for all $x,y\in \dom (T)$. 
 
 A symmetric operator $T$ is said to be bounded below by $M\in \RR$ if $\langle Tx,x\rangle \geq M \langle x,x\rangle$ holds for all $x\in \dom (T)$. $T$ is called (strictly) positive if $T$ is bounded below by a (strictly) positive number.
 Operators that are  bounded below by some number are also called semibounded. If $T$ is bounded below by $M$ and closable, then $\b T$ is also bounded below by $M$.
 
 A densely defined operator $T$ is called essentially self-adjoint if $\b T$ is self-adjoint. Proving that an operator is self-adjoint can be difficult. Below we will make use of the following criterion.

\begin{theorem}[{\cite[Theorem X.26]{ReedSimon:II}}]\label{thm:saExt}
Let $T$ be a strictly positive symmetric operator. Then the following are equivalent:
\begin{enuma}
	\item $T$ is essentially self-adjoint.
	\item $T$ has dense image.
	\item $T$ has only one semibounded self-adjoint extension. 
\end{enuma}
\end{theorem}

For a closed operator $T$ one  defines its spectrum as follows. A complex number $\lambda$ is said to be in the resolvent set of $T$ if $\lambda -T \colon \dom(T)\to \scrH$ is a bijection. The resolvent set is denoted $\rho(T)$. The complement of the resolvent set is the spectrum of $T$ and is denoted by $\sigma(T)$.

The following lemma is a simple consequence of the spectral theorem for unbounded operators (see \cite[Theorem VIII.4]{ReedSimon:I} or  \cite[Theorem 12.32]{Rudin:FA}).

\begin{lemma} \label{lemma_bound_and_spectrum}
	Let $T$ be a self-adjoint operator on $\scrH$. Then  $T$  is bounded below by $M$ if and only if $\sigma(T)\subset [M,\infty)$. 
\end{lemma}

A basic problem in perturbation theory is to show that $A+B$ inherits the properties of $A$ if $B$ is not too large compared to $A$.  This idea of a ``small perturbation" is made precise in the following definition.
Let $A$ and $B$ be densely defined operators on a Hilbert space $\scrH$. The operator $B$ is said to be $A$-bounded if $\dom(A) \subset \dom(B)$ and if there exist numbers $a,b\geq 0$ such that
\begin{equation}
	\label{eq:relatively_bounded}
	\|Bx\| \leq a \| Ax\| + b\| x\|
\end{equation} for all $x\in \dom(A)$.

\begin{theorem}[Kato--Rellich theorem, {\cite[Theorem X.12]{ReedSimon:II}}]
Suppose that $A$ is self-adjoint, $B$ is symmetric, and $B$ is $A$-bounded with $a<1$. Then $A+B$ is self-adjoint on $\dom(A)$. Further, if $A$ is bounded below by $M$, then $A+B$ is bounded below by $M- \max\{ b/(1-a), a|M|+b\}$, where $a$ and $b$ are given by \eqref{eq:relatively_bounded}. In particular, if $A$ is strictly positive and $b=0$, then $A+B$ is strictly positive.
\end{theorem}

We are going to apply the Kato--Rellich theorem to  very special perturbations.  The most transparent way to present our argument is to formulate a general statement.

\begin{lemma}[Homotopy lemma]\label{lemma:homotopy}
	Let $T\colon \dom(T)\to \mathscr{H}$ be an essentially self-adjoint, strictly positive operator on a Hilbert space $\mathscr{H}$. Let $S_t\colon \mathscr{H}\to \mathscr{H}$, $t\in [0,1]$, be a family of bounded operators with the following properties:
	\begin{enuma}
		\item $S_0=\id$;
		\item $t\mapsto S_t$ is continuous in the operator norm;
		\item each $S_t$ is invertible;
		\item the operators $T_t:=S_t \circ T$ are symmetric.
	\end{enuma}
Then $T_t$ is for every $t\in [0,1]$ an essentially self-adjoint, strictly positive operator satisfying $\dom(\b T_t)= \dom(\b T)$. 
\end{lemma}

\begin{proof} First of all observe that, since each $S_t$ is invertible,  $t\mapsto S_t^{-1}$ is also continuous in the operator norm. 
	Let $C:=\max_{t \in [0,1]} \|S_t^{-1}\|<\infty$. By continuity of $t \mapsto S_t$ on the compact interval $[0,1]$, we find $\epsilon>0$ such that
	\begin{displaymath}
		\|S_{t_1}-S_{t_2}\| \leq \frac{1}{2 C}
	\end{displaymath}
	whenever $|t_1-t_2| < \epsilon$.
	
	Since $\id \oplus  S_t $ is an invertible operator on  $\scrH\times \scrH$, the closures satisfy $\b T _t = S_t \circ \b T$ and $\dom(\b T_t)= \dom(\b T)$. 	
	Since $\b T_{t_1}-\b T_{t_2}=(S_{t_1}-S_{t_2}) \circ S_{t_1}^{-1} \circ \b T_{t_1}$, we obtain for $|t_1-t_2|<\epsilon$ and $x \in \dom (\b T)$ that
	\begin{align*}
		\|(\b T_{t_1}-\b T_{t_2})x\| & \leq \|S_{t_1}-S_{t_2}\| \cdot \|S_{t_1}^{-1}\| \cdot \|\b T_{t_1}x\| \leq \frac12 \|\b T_{t_1}x\|.
	\end{align*}
	Hence $\b T_{t_2}$ is a small perturbation of $\b T_{t_1}$ in the sense of the Kato--Rellich theorem, with $a=\frac12$ and $b=0$.  Thus $\b T_{t_2}$ is self-adjoint and strictly positive if $\b T_{t_1}$ has these properties.  Iterating this argument finitely many times we see that each $\b T_t$ is self-adjoint and strictly positive since  $\b T=\b T_0$ has these properties.
\end{proof}

\subsection{Primitive valuations}

Assume $1\leq k\leq n/2$ and let us fix a tuple $\bC=(C_0,C_1,\ldots, C_{n-2k})$ of convex bodies from $\calK_+^\infty(\RR^n)$.

\begin{definition}
The subspace of primitive valuations with respect to the reference bodies $\bC$, denoted $\VP_\bC\subset \VE_{n-k}$, is the kernel of $L_\bC\colon \VE_{n-k}\to \VE_{k-1}$. Analogously, the subspaces of primitive valuations in $\VF_{n-k}$ and $\Val^\infty_{n-k}$ are defined. We use superscripts $\VE, \VF, \infty$ if the ambient space is not obvious from the context.
\end{definition}

We define a projection onto the primitive subspace that is in general not orthogonal, but has other remarkable properties. Its existence  relies on the hard Lefschetz theorem for valuations proved in Sections \ref{s:SphereBundle} and \ref{s:Surjectivity}.

\begin{definition} We define the operator
	$\pi_{\bC}\colon \VE_{n-k}\to \VE_{n-k}$ by
\begin{displaymath}
 \pi_{\bC} = \id - L_{C_0} \circ L_{\bC,C_0}^{-1} \circ  L_\bC,
 \end{displaymath}
where $(\bC,C_0)$ denotes the   $(n-2k+2)$-tuple $(C_0,C_1,\ldots,C_{n-2k},C_0)$.
\end{definition}

For the spaces $\VF_{n-k}$ and $\Val_{n-k}^\infty$, the operator $\pi_{\bC}$ is defined by the same formula. Note that $\pi_{\bC}$ is in all cases continuous. We collect several elementary, but important properties.

\begin{lemma}\label{lemma:Pproj}
	
\begin{enuma}
\item The operator $\pi_{\bC}$ is a projection with image $\VP_\bC$. Its kernel is the image of $L_{C_0}\colon \VE_{n-k+1}\to\VE_{n-k}$.
\item Let $\bC'$ be another $(n-2k+1)$-tuple of reference bodies. The restriction $\pi_{\bC}\colon \VP_{\bC'}\to  \VP_{\bC}$ is an isomorphism.
\end{enuma}
The analogous statements also hold  for the spaces $\VF_{n-k}$ and $\Val^\infty_{n-k}$.
\end{lemma}

\begin{proof}
(a) is a formal consequence of the definition. 

For the proof of (b), let $\phi\in \VP_{\bC'}$ with $\pi_{\bC} \phi=0$. Then $\phi=L_{C_0}\psi$ for some $\psi\in \VE_{n-k+1}$. Hence $0=L_{\bC'} \phi= L_{\bC',C_0} \psi$, and thus $\psi=0$ by  the hard Lefschetz theorem applied to the $(n-2k+2)$-tuple $(\bC',C_0)$. This shows that the restriction $\pi_{\bC}\colon \VP_{\bC'}\to  \VP_{\bC}$ is injective.

To prove surjectivity, let $\psi\in \VP_\bC$. By the hard Lefschetz theorem there exists $\eta\in \VE_{n-k+1}$ satisfying
\begin{displaymath}
	L_{\bC',C_0} \eta = L_{\bC'} \psi.
\end{displaymath}
Put $\phi = \psi -L_{C_0} \eta\in \VE_{n-k}$. Note that $L_{\bC'} \phi =0 $ and hence $\phi\in \VP_{\bC'}$. Since the image of $L_{C_0}$ is the kernel of $\pi_{\bC}$, we obtain $\pi_{\bC} \phi = \pi_{\bC}\psi = \psi$,  as required.  The proof for the spaces $\VF_{n-k}$ and $\Val^\infty_{n-k}$ is verbatim the same.
\end{proof}

Let us explicitly state the following consequence of the lemma.
\begin{lemma} $\VP^\infty_\bC$ is a dense subspace of $\VP^\VE_\bC$.
\end{lemma}
\begin{proof}
Observe that by definition each Lefschetz map $L_C$ preserves the subspace of smooth valuations. 	
Since $\Val^\infty_{n-k}$ is dense in $\VE_{n-k}$ by the definition of the latter space and  the image of $\pi_\bC$ is $\VP_\bC^\VE$, the claim follows from the fact that $\pi_\bC$ preserves the subspace of smooth valuations.
\end{proof}

An immediate consequence of Proposition~\ref{prop:contDependence} and Corollary~\ref{cor_contDependence_E}
is

\begin{proposition}\label{prop:contDependencePC} The map
\begin{displaymath}
	\pi: \calK^\infty_+(\RR^n)^{n-2k+1}   \to \mathscr{L}(\VE_{n-k}), \quad \bC \mapsto \pi_{\bC}\\
\end{displaymath}
is continuous, where the domain is equipped with the $C^2$ topology and the target is equipped with the operator norm.
	 The same conclusion holds if $\VE$ is replaced by $\VF$.
\end{proposition}

\subsection{Proof of the Hodge--Riemann relations}

Our strategy will be to fit the Hodge--Riemann relations for smooth valuations into a framework where we can apply  the Kato--Rellich theorem in the form of the homotopy lemma (Lemma~\ref{lemma:homotopy}).

Let $\bB$ denote the tuple of reference bodies $(B^n,\ldots,B^n)$, where $B^n$ is the euclidean unit ball.
Consider on $\VP_\bB^\VE$ the densely defined quadratic forms
\begin{align*}
 q_\bB(\phi,\psi) & =(-1)^k \phi * L_{\bB_{\setminus0}}\bar \psi, \\
 \pi_\bC^* q_\bC(\phi,\psi) &  =q_\bC(\pi_\bC \phi,\pi_\bC \psi) =(-1)^k \pi_\bC \phi * (L_{\bC_{\setminus0}} \circ  \pi_\bC \bar \psi),
\end{align*}
for $ \phi,\psi \in \VP_\bB^\infty$.

\begin{lemma} \label{lemma:HR_TCpositive}
	The Hodge--Riemann relations are satisfied for $\bC$ if and only if $\pi_\bC^* q_\bC(\phi,\phi) >0$ for all $\phi \in \VP^\infty_\bB, \phi \neq 0$.
\end{lemma}

\begin{proof}
This is clear, since $\pi_\bC:\VP^\infty_\bB \to \VP^\infty_\bC$ is bijective.
\end{proof}

By Proposition \ref{prop:SobolevVal} and Theorem \ref{thm:poincare}, the forms $q_\bB$ and $\pi_\bC^* q_\bC$ can be extended to continuous maps $\VP_\bB^\infty \times \VP_\bB^\VE \to \C$. Since $\VP_\bB^\VE$ is a closed subspace of the Hilbert space $\VE_{n-k}$, it is a Hilbert space itself. By the Riesz representation theorem we may thus write
\begin{align*}
 q_\bB(\phi,\psi) & = \langle T_\bB \phi, \psi\rangle_\VE\\
 \pi_\bC^* q_\bC(\phi,\psi) & = \langle T_\bC \phi, \psi\rangle_\VE
\end{align*}
with symmetric operators $T_\bB$ and $T_\bC$ on $\VP_\bB^\VE$ that are densely defined on $\VP_\bB^\infty$.

The next proposition is crucial in our proof.
\begin{proposition}\label{prop:HRmain} There exists a bounded invertible operator $S_\bC\colon  \VP_\bB^\VE \to \VP_\bB^\VE$ such that
	$T_\bC = S_\bC \circ T_\bB$. Explicitly,
	\begin{equation}\label{eq:SC}
		S_\bC:=\left(\pi_\bB \circ L_{\bB_{\setminus0}}^{-1} \circ L_{\bC_{\setminus0}} \circ \pi_\bC\right)^*: \VP_\bB^\VE \to \VP_\bB^\VE.
	\end{equation}
\end{proposition}

\begin{proof}
We compute for any $\phi,\psi\in\VP_\bB^\infty$
\begin{align*}
 \langle T_\bC\phi,\psi\rangle_\VE & = \pi_\bC^* q_\bC(\phi,\psi) \\
 & =(-1)^k \pi_\bC \phi * L_{\bC_{\setminus0}} \circ \pi_\bC \bar \psi\\
 & =(-1)^k \phi * L_{\bC_{\setminus0}} \circ \pi_\bC \bar \psi\\
 & = (-1)^k \phi * L_{\bB_{\setminus0}} \circ L_{\bB_{\setminus0}}^{-1} \circ L_{\bC_{\setminus0}} \circ \pi_\bC \bar \psi\\
 & = (-1)^k \phi * L_{\bB_{\setminus0}} \circ \pi_\bB \circ L_{\bB_{\setminus0}}^{-1} \circ L_{\bC_{\setminus0}} \circ \pi_\bC \bar \psi\\
  & = q_\bB(\phi, \pi_\bB \circ L_{\bB_{\setminus0}}^{-1} \circ L_{\bC_{\setminus0}} \circ \pi_\bC \psi)\\
   & = q_\bB(\phi, S_\bC^* \psi)\\
   & = \langle S_\bC \circ T_\bB \phi,\psi\rangle_\VE.
\end{align*}
In the third line we use that $\mathrm{im}(\pi_\bC-\mathrm{id})\subset\mathrm{im}(L_{C_0})$ by Lemma \ref{lemma:Pproj}(a), and that $L_{C_0} \circ L_{\bC_{\setminus0}} \circ \pi_\bC \bar \psi=L_\bC \pi_\bC \psi=0$ since $\pi_\bC \bar \psi$ is primitive with respect to $\bC$. Similarly, in the fifth line we use that $\mathrm{im}(\pi_\bB-\mathrm{id})\subset\mathrm{im}(L_{B^n})$, and $L_{B^n} \circ L_{\bB_{\setminus0}}  \phi=L_\bB \phi=0$ since $\phi$ is primitive with respect to $\bB$.
\end{proof}

\begin{proposition} \label{prop_positivity_tb_tc}
	\begin{enuma}
	\item The operator $T_\bB$ on $\VP_\bB$ is strictly positive and essentially self-adjoint, where $\bB=(B^n,\ldots,B^n)$ with $B^n$ the euclidean unit ball.
	\item The operator $T_\bC$ on $\VP_\bB$ is strictly positive and essentially self-adjoint, where $\bC \in \calK^\infty_+(\RR^n)^{n-2k+1}$ is an arbitrary tuple of reference bodies.
	\item The Hodge--Riemann relations hold for the reference bodies $\bC$.
	\end{enuma}
\end{proposition}

\begin{proof}

The proof of (a) is based on a computation involving highest weight vectors. Since this computation is rather lengthy and requires us to introduce further notation, we postpone it to the final section. 

Let us prove (b). For $t\in[0,1]$ and  $i=0,\ldots, n-2k$ we define  
\begin{displaymath}
  C_i(t) := (1-t)B^n + t C_i \in \calK^\infty_+(\RR^n), \quad \bC(t):=(C_0(t), \ldots, C_{n-2k}(t)). 	
\end{displaymath}
Clearly $\bC(0)=\bB$ and $\bC(1)=\bC$. Keeping the notation of Proposition~\ref{prop:HRmain}, we define $S_t:=S_{\bC(t)}$, so that $T_{\bC(t)} = S_t \circ T_\bB$. By Proposition~\ref{prop:HRmain}, each $S_t$ is invertible. Moreover, in view of \eqref{eq:SC} and the continuous dependence on the reference bodies guaranteed by  Propositions~\ref{prop:contDependence} and \ref{prop:contDependencePC}, the map $t\mapsto S_t$ is continuous in the operator norm. Lemma \ref{lemma:HR_TCpositive} shows that $T_{\bC(t)}$ is symmetric. Since $T_\bB$ is essentially self-adjoint and strictly positive by (a), the hypothesis of the homotopy lemma is satisfied. We thus conclude that  $T_\bC$ is strictly positive and essentially self-adjoint. 

Part (c) follows from (b) and Lemma~\ref{lemma:HR_TCpositive}. 
\end{proof}

\begin{proof}[Proof of Theorem~\ref{thm:mainMV}]
 The assumption in Theorem \ref{thm:mainMV} says that
		\begin{displaymath}
			\phi:=\sum_{i=1}^N x_i V(A_1^i,\ldots,A_k^i,\Cdot[n-k]) \in \VP^\infty_\bC.
		\end{displaymath}
	We have
		\begin{equation} \label{eq:qc_explicit}
		\begin{split}
		q_\bC(\phi,\phi)=(-1)^k \frac{(n-k)!^2}{n!} \sum_{i,j=1}^N x_ix_j V(A_1^i,\ldots,A_k^i,A_1^j,\ldots,A_k^j,C_1,\ldots,C_{n-2k}),
			\end{split}
		\end{equation}
		by \eqref{eq_convolution_mixed_volumes} and the definition of $L_\bC$. The Hodge--Riemann relations imply $q_\bC(\phi,\phi)\geq0$ with equality if and only if $\phi=0$.  
\end{proof} 

\section{Spectral properties}

\label{s:Spectral}

In the final section we will investigate the properties of the Hodge--Riemann sesquilinear form $q_\bC$ as a quadratic form on the Hilbert spaces  $\VP_\bC^\VE$ and $\VP_\bC^\VF$. We will complete the proof of the Hodge--Riemann relations by a direct computation in the case where all reference bodies are euclidean balls, which is the starting point in the homotopy argument in the proof of Proposition \ref{prop_positivity_tb_tc}. We also clarify our choice of the Hilbert space $\VE$ instead of $\VF$ in this last section by showing that the properties of the Hodge--Riemann sesquilinear form on $\VE$ and $\VF$ are very distinct. The former is closable and has a discrete spectrum of eigenvalues with finite multiplicities accumulating at  $+\infty$, while the latter is bounded and has a spectrum of eigenvalues with finite multiplicities accumulating at $0$. Using the $\VE$ completion, we may  relax the smoothness assumption in Theorem~\ref{thm:mainMV}.

\subsection{Further facts about self-adjoint operators}

Our proofs will rely on certain properties of self-adjoint operators that, for the convenience of the reader, we review in this section.

An operator $A$ with the property that  $(A-\mu)^{-1}$ is compact for all $\mu\in \rho(A)$ is said to have compact resolvent.  We will need the following characterization of such operators. 

\begin{theorem}[{\cite[Theorem XIII.64]{ReedSimonIV}}]\label{thm:compactresolvent}
	Let $A$ be a semibounded self-adjoint operator on $\scrH$. Then the following are equivalent:
	\begin{enuma}
		\item $(A-\mu)^{-1}$ is compact for some $\mu\in \rho(A)$. 
		\item $(A-\mu)^{-1}$ is compact for all $\mu\in \rho(A)$. 
		\item $\{ x\in \dom(A)\colon \|x\|\leq 1 \text{ and } \| Ax\|\leq b\}$ is compact for all $b$. 
		\item \label{item:compact_eigenvalues} There exists a complete orthonormal basis $(x_n)_{n\in \NN}$ in $\dom(A)$ so that $Ax_n = \mu_n x_n$ with  $\mu_1\leq \mu_2\leq \cdots $ and $\mu_n\to \infty$. 
		\item The spectrum of $A$ is discrete and consists of eigenvalues with finite multiplicities. 
	\end{enuma}
\end{theorem}

Next, we need some facts about quadratic forms and refer to \cite[Section VIII.6]{ReedSimon:I} for details. By a quadratic form $q$ on a Hilbert space $\scrH$ we will understand a sesquilinear form $q\colon \dom(q)\times \dom(q)\to \CC$ defined on a dense subspace $\dom(q)$ of $\scrH$. If there is a number $M\in \RR$ such that  $q(x,x)\geq M \|x\|^2$ holds for all $x\in \dom (q)$, then $q$ is said to be semibounded.  If $\scrH$ is complex, this implies that $q$ is symmetric, i.e., $q(y,x)=\overline{q(x,y)}$ for all $x,y$. If $M\geq 0$, then $q$ is said to be positive. If $|q(x,x)| \leq M\|x\|^2$ for some $M$ and all $x \in \dom(q)$, then $q$ is called bounded.

Let $q$ be a quadratic form bounded below by $M$. Then $q$ is called closed if $\dom(q)$ is  complete under the norm
\begin{displaymath}
	\|x\|_{+1} = \sqrt{ q(x,x)+ (1-M) \|x\|^2}.
\end{displaymath}
We will need the following fact.

\begin{lemma}[{\cite[Problem 15, Chapter VIII]{ReedSimon:I}}]\label{lemma:closed_quadratic} 
	A semibounded quadratic form $q$ is closed if and only if whenever 
	\begin{displaymath}
	  x_i\in \dom(q), \quad x_i \stackrel{\scrH}{\to} x, \quad \text{and}\quad q(x_i-x_j, x_i-x_j)\underset{i,j\to \infty}{\to} 0,
	\end{displaymath}
	then $x \in \dom(q)$ and $q(x_i-x, x_i-x)\to 0$. If this is the case, we have $q(x,x)=\lim_{i \to \infty} q(x_i,x_i)$. 
	\end{lemma}

The closure of a quadratic form $q$, if it exists, is the smallest closed quadratic form which extends it. If a closure of $q$ exists, then $q$ is called closable.

\begin{theorem}[Friedrichs extension theorem {\cite[Theorem X.23]{ReedSimon:II}}] \label{thm_friedrich}
	Let $A$ be a positive symmetric operator and let $q(x,y)=\langle Ax,y\rangle$ for $x,y\in \dom (A)$. Then $q$ is a closable quadratic form and its closure $\hat q$ is the quadratic form of a unique self-adjoint operator $\hat A$. $\hat A$ is a positive extension of $A$ and the lower bound for the spectrum of $\hat A$ is the lower bound of $q$.
\end{theorem}

\subsection{Spectral properties  of the operator $T_\bB$}

The paper \cite{KotrbatyWannerer:Harmonic} constructs differential forms  $\omega_{r,k,m} \in X_r$ such that the induced valuations $\phi_{r,k,m}\in \Val_r^\infty(\RR^n)$ are highest weight vectors  for the natural action of the orthogonal group.
Moreover,  \cite[Theorem~1.4]{KotrbatyWannerer:Harmonic} computes $q_\bB(\phi_{r,k,m})$. To prove Proposition~\ref{prop_positivity_tb_tc}(a), and thereby to complete the proof of the Hodge--Riemann relations, it will be sufficient to compute the $L^2$ norm of the differential forms $ \omega_{r,k,m}$. Let $s_n$ denote the volume of the $n$-dimensional unit sphere and let $n v_n = s_{n-1}$ denote the volume of the $n$-dimensional euclidean unit ball. 
\begin{lemma} The valuation $\phi_{r,k,m}\in \Val_r^\infty$ of \cite[Theorem~1.3]{KotrbatyWannerer:Harmonic} satisfies
	\begin{displaymath}
	  \| \phi_{r,k,m}\|_{\VE}^2  \leq
	 \frac{(n+m-r-2)s_{n+2m-3}}{s_{n+m-r-3}^2s_{2m-3}}{n-2k\choose r-k}.
	\end{displaymath}
\end{lemma}

\begin{proof}
	Using the notation introduced on page 560 of \cite{KotrbatyWannerer:Harmonic}, but moving the constant  appearing in front of the integral in \cite[Theorem~1.3]{KotrbatyWannerer:Harmonic} into the definition of $\omega_{r,k,m}$, we have that 
	\begin{displaymath}
	\omega_{r,k,m}=\frac{(\sqrt{-1})^{\lfloor\frac n2\rfloor}(\sqrt 2)^{m-2}}{s_{n+m-r-3}}\zeta_{\b1}^{m-2}(\omega_{r,k}^{(1)}+\omega_{r,k}^{(2)}),
	\end{displaymath}	
	where, using the multinomial theorem, 
	\begin{align*}
		\omega_{r,k}^{(1)}\otimes\Theta_1&=\zeta_{\b K}(d\zeta_{\b K})^{[k-1]}(d\zeta_L)^{[n-r-k]}(dz_L)^{[r-k]}(\b{dz_K})^{[k]}\\
		&=\sum_{a=1}^k\sum_{\substack{A\subset L\\|A|=r-k}} (\zeta_{\b a}\otimes dz_{\b a})(d\zeta_{\b {K_a}})^{[k-1]}(d\zeta_{L\setminus A})^{[n-r-k]}(dz_A)^{[r-k]}(\b{dz_K})^{[k]}\\
		&=:\sum_{a=1}^k\sum_{\substack{A\subset L\\|A|=r-k}} \omega^{(1)}_{a,A}\otimes\Theta_1,
	\end{align*}
	and
	\begin{align*}
		\omega_{r,k}^{(2)}\otimes\Theta_1&=\zeta_{L}(d\zeta_{\b K})^{[k]}(d\zeta_L)^{[n-r-k-1]}(dz_L)^{[r-k]}(\b{dz_K})^{[k]}\\
		&=\sum_{a\in L}\sum_{\substack{A\subset L_a\\|A|=r-k}}(\zeta_a\otimes dz_a)(d\zeta_{\b K})^{[k]}(d\zeta_{L_a\setminus A})^{[n-r-k-1]}(dz_A)^{[r-k]}(\b{dz_K})^{[k]}\\
		&=:\sum_{a\in L}\sum_{\substack{A\subset L_a\\|A|=r-k}} \omega^{(2)}_{a,A}\otimes\Theta_1.
	\end{align*}
	Here $L_a=L\setminus\{a\}$. Since $\omega^{(1)}_{a,A},\omega^{(2)}_{a,A}$ are both multiples of an orthonormal basis vector, it is readily seen that pointwise
	\begin{displaymath}
		\langle\zeta_{\b1}^{m-2}\omega^{(i)}_{a,A},\zeta_{\b1}^{m-2}\omega^{(j)}_{b,B}\rangle=\delta_{i,j}\delta_{a,b}\delta_{A,B}|\zeta_{\b1}|^{2m-4}|\zeta_a|^2.
	\end{displaymath}
	 
 Integration in spherical coordinates gives
	\begin{displaymath}
		\int_{S^{n-1}}|\zeta_1|^{2m-4}\sum_{a=1}^k|\zeta_a|^2\,d\sigma
			=\frac{(m+k-2)}{2^{m-2}}\frac{s_{n+2m-3}}{s_{2m-3}}
	\end{displaymath}
	and similarly
	\begin{displaymath}
		\int_{S^{n-1}}|\zeta_1|^{2m-4}\sum_{a\in L}|\zeta_a|^2\,d\sigma
	     =\frac{(n-2k)}{2^{m-2}}\frac{s_{n+2m-3}}{s_{2m-3}}.
	\end{displaymath}
	Since
\begin{displaymath}
  (\omega_{r,k}^{(1)}+\omega_{r,k}^{(2)})\otimes\Theta_1=\zeta_J(d\zeta_J)^{[n-r-1]}(dz_J)^{[r-k]}(\b{dz_K})^{[k]}
\end{displaymath}
we clearly have $i_E\omega_{r,k,m}=0$, where $E$ is the Euler vector field.
	Altogether,
	\begin{align*}
		&\|\omega_{r,k,m}\|^2_{L^2}\\
		&\quad=\int_{S^{n-1}} |\omega_{r,k,m}|^2-|i_E\omega_{r,k,m}|^2\, d\sigma\\
		&\quad=\frac{2^{m-2}}{s_{n+m-r-3}^2}\int_{S^{n-1}}|\zeta_1|^{2m-4}\left[{n-2k\choose r-k}\sum_{a=1}^k|\zeta_a|^2+{n-2k-1\choose r-k}\sum_{a\in L}|\zeta_a|^2\right]\, d\sigma\\
		&\quad=\frac{(n+m-r-2)s_{n+2m-3}}{s_{n+m-r-3}^2s_{2m-3}}{n-2k\choose r-k}.
	\end{align*}
Since $\|\phi_{r,k,m}\|_\VE = \inf_{\eta\in Y_r} \| \omega_{r,k,m} + \eta\|_{L^2}\leq \|\omega_{r,k,m}\|_{L^2} $, we obtain the desired inequality. 
\end{proof}

\begin{proof}[Proof of Proposition~\ref{prop_positivity_tb_tc}(a)]
 	The decomposition of the space $\VP_{\bB}\subset \VE_{n-r}$ as an orthogonal sum of irreducible representations under the action of the orthogonal group $\mathrm{SO}(n)$ was obtained in \cite{KotrbatyWannerer:Harmonic}, based on \cite{ABS:Harmonic}: 
	\begin{equation}\label{eq:L2decomp} 
		\VP_{\bB} = \b{\bigoplus_\lambda \Gamma_\lambda},
	\end{equation}
where $\lambda$ runs over all highest weights of the form
\begin{displaymath}
	 (m,\underbrace{2,\ldots,2}_{r-1\ \text{times}},0,\ldots,0), \quad \text{ together with } (m,\underbrace{2,\ldots,2}_{r-2\ \text{times}},-2) \text{ if } n=2r,
\end{displaymath}
and where $\Gamma_\lambda$ is the irreducible representation of $\mathrm{SO}(n)$ with highest weight $\lambda$.
Moreover, $\Gamma_\lambda\subset \Val^\infty_r$.	

Since $T_\bB$ is $\SO(n)$-equivariant, by Schur's lemma, $\Gamma_\lambda$ is an eigenspace of $T_\bB$. Let  $c_\lambda$ denote the corresponding (necessarily real) eigenvalue.  The valuation $\phi_{n-r,r,m}$ is a highest weight vector in $\VP_{\bB}$ with highest weight $\lambda=(m,2,\dots,2,0,\dots,0)$ and
\begin{align*}
	\langle T_\bB\phi_{n-r,r,m},\phi_{n-r,r,m}\rangle=(n-2r)!\frac{(m+r-1)(n+m-r)v_{n+2m-2}}{v_{r+m-2}^2 s_{2m-3}}
\end{align*}
by \cite[Theorem 1.4]{KotrbatyWannerer:Harmonic}. 	Thus, using the previous lemma,
	\begin{align*}
	\frac{\langle T_\bB\phi_{n-r,r,m},\phi_{n-r,r,m}\rangle}{\|\phi_{n-r,r,m}\|_\VE^2}\geq\frac{(n-2r)!(n+m-r)(r+m-1)(r+m-2)}{(n+2m-2)}.
	\end{align*}
	Since  the number on the right-hand side is $\geq1$, we conclude that $c_\lambda\geq 1$. Similarly one proceeds for the highest weight $\lambda=(m,2,\dots,2,-2)$.
	
	If $\phi$ is any element of $\VP_\bB$ and we decompose $\phi = \sum_\lambda \phi_\lambda$ according to \eqref{eq:L2decomp}, then  
	\begin{displaymath}
	   \langle  T_\bB \phi, \phi\rangle = \sum_\lambda c_{\lambda} \|\phi_\lambda\|^2_\VE\geq \|\phi\|^2_\VE.
	\end{displaymath}
	Hence $T_\bB$ is bounded below by $1$ and in particular strictly positive. 
	
	Since each $c_\lambda$ is non-zero, the decomposition~\eqref{eq:L2decomp} also shows that $T_\bB$ has dense image. Since $T_\bB$ is symmetric, Theorem~\ref{thm:saExt} implies that $T_\bB$ is essentially self-adjoint.	
\end{proof}

\subsection{The Hodge--Riemann form as a  quadratic form on a Hilbert space}

The theory of quadratic forms on an infinite-dimensional space is much richer than on a finite-dimensional space. In this final section we consider the Hodge--Riemann sesquilinear form $q_\bC$ from this perspective. In particular,  we prove Theorem~\ref{thm:mainE} from the introduction as a part of Theorem~\ref{thm:mainQuadratic} below.

\begin{lemma}\label{lemma:TCcompact}
For every tuple $\bC$ of reference bodies from $\calK_+^\infty(\RR^n)$, the operator $\b T_\bC$ has compact resolvent. Moreover, $0$ belongs to the resolvent set.
\end{lemma}

\begin{proof}
That $\b T_\bB$ has compact resolvent follows from the proof of Proposition~\ref{prop_positivity_tb_tc}(a) using Theorem~\ref{thm:compactresolvent}\ref{item:compact_eigenvalues}\hspace{-.3em}. Recall that $\b T_\bC= S_\bC\circ \b T_\bB$, where $S_\bC$ is an isomorphism. 
It follows that for every $b$ 
\begin{equation}\label{eq:compTA} \{ \phi \in \dom \b T_\bC\colon  \|\phi\|\leq 1 \text{ and } \| \b T_\bC \phi\|\leq b\}\end{equation}
is contained in
\begin{displaymath}
 \{ \phi \in \dom \b T_\bB\colon  \|\phi\|\leq 1 \text{ and } \| \b T_\bB \phi\|\leq b\| S_\bC^{-1}\|\}.
\end{displaymath}
Since $\b T_\bB$ has compact resolvent, the latter set is compact by Theorem~\ref{thm:compactresolvent}.

 Since \eqref{eq:compTA} is closed, it follows that it is also compact. Using again Theorem~\ref{thm:compactresolvent} it follows that $\b T_\bC$ has compact resolvent. Since $T_\bC$ is strictly positive by Proposition  \ref{prop_positivity_tb_tc}(b), so is $\b T_\bC$ and hence  $\sigma(\b T_\bC) \subset (0,\infty)$ by Lemma~\ref{lemma_bound_and_spectrum}.
\end{proof}

\begin{theorem}\label{thm:mainQuadratic}
	Let  $\bC=(C_0,C_1,\ldots, C_{n-2k})$ be a tuple of  reference bodies from $\calK_+^\infty(\RR^n)$. 
	\begin{enuma}
		\item 
	The Hodge--Riemann sesquilinear form $q_\bC$ defined on $\VP_\bC^\infty\subset \VE_{n-k}$ is closable. Its closure $q^\VE_\bC$ is the quadratic form of a unique  self-adjoint operator $Q_\bC^\VE$.  This operator is strictly positive. Its spectrum is  discrete and consists of eigenvalues with finite multiplicities.
	\item The Hodge--Riemann sesquilinear form $q_\bC$ defined on $\VP_\bC^\infty\subset \VF_{n-k}$ is bounded. Its closure $q^\VF_\bC$ is the quadratic form of a unique bounded  self-adjoint operator $Q_\bC^\VF$. This operator is positive and compact. Consequently, each non-zero element of the spectrum is an eigenvalue with finite multiplicity and the eigenvalues accumulate only at $0$.  
	\end{enuma}
\end{theorem}

\begin{proof} (a) Note that
\begin{displaymath}
q_\bC(\phi) = \langle \pi_\bC^{-*} \circ T_\bC \circ \pi_\bC^{-1}  \phi, \phi\rangle_\VE, \quad \phi\in \VP_\bC^\infty.
\end{displaymath}
That $q_\bC$ is closable follows from the Hodge--Riemann relations and Friedrichs extension theorem (Theorem~\ref{thm_friedrich}). Its closure corresponds to a unique self-adjoint operator $Q_\bC^\VE$ and this operator extends  $\pi_\bC^{-*} \circ T_\bC \circ \pi_\bC^{-1} $. But $\pi_\bC^{-*} \circ \b T_\bC \circ \pi_\bC^{-1}$ is another positive self-adjoint extension. By Theorem~\ref{thm:saExt}, these extensions must coincide: $Q_\bC^\VE=\pi_\bC^{-*} \circ \b T_\bC \circ \pi_\bC^{-1}$. By Lemma~\ref{lemma:TCcompact}, $\b T_\bC$ has compact resolvent. In particular, since zero belongs to the resolvent set of $\b T_\bC$ by Lemma \ref{lemma:TCcompact},  $\b T_\bC^{-1}$ is compact. It follows that $(Q_\bC^\VE)^{-1}$ is compact and Theorem~\ref{thm:compactresolvent} shows that $Q_\bC^\VE$ has compact resolvent. This implies the statement about the spectrum of $Q_\bC^\VE$.
	
	(b) Recall from Proposition~\ref{prop:SobolevVal} that the inclusion $\iota_k \colon \VF_{k}\hookrightarrow \VE_k$ is compact. Let $r_{n-k}:\VF_{n-k}^* \to \VF_{n-k}$ be the antilinear map induced by the Riesz representation theorem, i.e., $\langle \phi,r_{n-k} \tau \rangle_\VF=\langle \phi,\tau \rangle, \tau \in \VF_{n-k}^*, \phi \in \VF_{n-k}$.
	
	For $\phi,\psi\in \VP^\infty_\bC\subset \VF_{n-k}$ we have 	
	\begin{align*}
		q_\bC(\phi,\psi) & =(-1)^k \phi * L^\VF_{\bC_{\setminus0}} \b \psi\\
		& = (-1)^k \langle \phi,\Pd_k \circ \iota_k \circ L_{\bC_{\setminus0}}^\VF \b \psi\rangle \\
		& =(-1)^k \langle \phi,r_{n-k} \circ \Pd_k \circ \iota_k \circ L_{\bC_{\setminus0}}^\VF \b \psi\rangle_\VF \\
		& = \langle \phi, Q_\bC^\VF \psi \rangle_\VF,
	\end{align*}
where 
\begin{displaymath}
	Q_\bC^\VF \psi= (-1)^k r_{n-k} \circ \Pd_k \circ \iota_k \circ L_{\bC_{\setminus0}}^\VF \b \psi.
\end{displaymath}

	Since the composition of a compact operator with a bounded operator yields a compact operator, we conclude that $Q_\bC^\VF$ is bounded, self-adjoint, and compact. The Hodge--Riemann relations for smooth valuations imply that $Q_\bC^\VF$ is positive. The statement about the spectrum is a well-known property of compact operators, see, e.g., \cite[Theorem VI.15]{ReedSimon:I}.
\end{proof}

 We say that a convex body $K$ in $\RR^n$ belongs to the class $\calK^{1,1}(\RR^n)$ if its  support function is differentiable on $\RR^n\setminus \{0\}$ and  all partial derivatives $\partial_i h_K\colon \RR^n\setminus\{0\} \to \RR$  are locally Lipschitz. Recall that a function $f\colon A\to \RR$, where $A\subset \RR^n$, is called locally Lipschitz if its restriction to compact subsets of $A$ is Lipschitz. Moreover, if $U$ is an open subset of $\RR^n$, then $f$ is locally Lipschitz in $U$ if and only if $f\in W^{1,\infty}_{\mathrm{loc}}(U)$, see \cite[Section 4.2.3, Theorem 5]{EvansGariepy:Measure}. Recall also that a function $f$ belongs to the Sobolev space $ W^{1,\infty}_{\mathrm{loc}}(U)$ if and only if $
f\in W^{1,\infty}(V)$ for each compactly contained open subset $V\subset  U$.

\begin{lemma} \label{lemma:C11}
Let $A\in \calK^{1,1}(\RR^n)$. Then there exists a sequence  $(A_l)$ of  convex bodies from $\calK_+^\infty(\RR^n)$ such that 
\begin{enuma}
\item $h_{A_l}\to h_A$ in $C^1(\RR^n\setminus\{0\})$ as $l\to \infty$, and 
\item the partial derivatives 
\begin{displaymath}
	\left. \frac{\partial^2h_{A_{l}}}{\partial x_i\partial x_j}\right|_{S^{n-1}}
\end{displaymath}
 converge in $L^\infty(S^{n-1})$. 
 \end{enuma}
Moreover, the limit in (b) does not depend on the approximating sequence $(A_l)$ and is denoted by $\left. \frac{\partial^2h_{A}}{\partial x_i\partial x_j}\right|_{S^{n-1}}.$
\end{lemma}
\begin{proof}
The standard mollification procedure, see \cite[Theorem 3.4.1]{Schneider:BM} and the comments after this theorem concerning $\calK_+^\infty(\RR^n)$, yields a sequence $(A_l)$ of bodies from $\calK_+^\infty(\RR^n)$ satisfying (a).  	For 
	$V=\{x\in \RR^n\colon \frac 12 < |x| \leq 2\}$  we have $\partial_i h_A\in W^{1,\infty}(V)$ and $\partial_i h_{A_l} \to \partial_i h_A$ in $W^{1,\infty}(V)$. 
	Since the second partial derivatives are homogeneous functions, it follows that 
	\begin{displaymath}
	   \left. \frac{\partial^2h_{A_{l}}}{\partial x_i\partial x_j}\right|_{S^{n-1}} 
	\end{displaymath}
	converges in $L^\infty(S^{n-1})$ and that the limit does not depend on the sequence $(A_l)$.
\end{proof}

In view of the definition \eqref{eq:omegaOmega} and the preceding lemma we may define a differential form with $L^\infty$ coefficients $\omega_\bA\in \Omega_{L^\infty}^{n-k,k-1}(S\RR^n)^{tr} = L^\infty(E_{n-k,k-1})$, 
\begin{displaymath}
  \omega_{\mathbf A} := \frac{(n-k)!}{n!}  \cdot   *_1^{-1}(*_1\omega_{A_1} \wedge \Omega_{A_{2}} \wedge \cdots \wedge \Omega_{A_{k}}),
\end{displaymath}
for  each tuple  $\bA= (A_1,\ldots, A_{k})$  of convex bodies from $\calK^{1,1}(\RR^n)$ with $ 1 \leq k \leq n$. Let $\phi_\bA\in \VE_{n-k}$ denote the corresponding valuation.

In the following lemma we use  that generalized valuations can be evaluated on convex bodies belonging to the class $\calK_+^{\infty}(\RR^n)$ and that  $\VE\subset \Val^{-\infty}$ by Proposition~\ref{prop:SobolevVal}. 
\begin{lemma} \label{lemma:C11_mixedvolume} 
	For each tuple  $
\bA= (A_1,\ldots, A_{k})$ of convex bodies from $\calK^{1,1}(\RR^n)$, the valuation  $\phi_{\bA}\in \VE_{n-k}$ is uniquely determined by the property 
\begin{displaymath}
	\phi_{\bA}(K)= V(A_1,\dots,A_{k}, K[n-k] )
\end{displaymath}
for all $K\in \calK_+^\infty(\RR^n)$.  
Moreover, if $ 1 \leq k \leq \frac{n}{2}$ and $\bC=(C_0,\ldots, C_{n-2k})$ is a tuple of convex bodies from $\calK^\infty_+(\RR^n)$, then
\begin{displaymath}
	 L_\bC\phi_\bA(K)= \frac{(n-k)!}{(k-1)!}V( A_1,\dots,A_{k}, C_0,\ldots, C_{n-2k},K[k-1])
\end{displaymath}
 for all $K\in \calK_+^\infty(\RR^n)$.
\end{lemma}

\begin{proof}
Choose approximating sequences $(A_{j,l})$ as in Lemma~\ref{lemma:C11} and let $\bA_l=(A_{1,l}, \ldots, A_{k,l})$. Then $ \omega_{\bA_l} \to\omega_{\bA}$ in $L^\infty$ as $l\to \infty$ and 
 by Lemma~\ref{lemma:form_mixed_volume} and Theorem~\ref{thm:BFconvolution} one has $\phi_{\bA_l}= V( A_{1,l},\dots,A_{k,l},\Cdot[n-k])$.
Therefore  \eqref{eq:evaluation} and  Corollary~\ref{cor:evaluation} imply  
 \begin{align*}
 \phi_\bA(K) &=   \frac{(-1)^{\binom{n}{2}}}{(n-k)!}  \int_{S^{n-1}} *_1\omega_{\bf A}\wedge  (\Omega_K)^{n-k}   \\
 &= \lim_{l\to\infty}   \frac{(-1)^{\binom{n}{2}}}{(n-k)!}  \int_{S^{n-1}} *_1\omega_{\bA_l}\wedge  (\Omega_K)^{n-k}\\
 &=  \lim_{l\to\infty} V( A_{1,l},\dots,A_{k,l},K[n-k])\\
 & =  V( A_1,\dots,A_{k},K[n-k]).
 \end{align*}
The same argument shows the statement concerning $L_\bC\phi_\bA(K)$.
\end{proof}

\begin{proposition}\label{prop:C11}
Let $0\leq k\leq \frac n2$, $N\in\NN$, $x_i\in\RR$, $A_j^i\in\calK^{1,1}(\RR^n)$, and $C_l\in \calK^\infty_+(\RR^n)$. If 
\begin{displaymath}
   \phi = \sum_{i=1}^N x_i V(A_1^i,\dots,A_k^i,\Cdot[n-k]) \in \VE_{n-k}
\end{displaymath}
is $\bC$-primitive, then $\phi$ lies in the domain of $q_\bC^\VE$ and
$$
	q_\bC^\VE(\phi,\phi) =  (-1)^k  \frac{(n-k)!^2}{n!}\sum_{i,j=1}^N x_ix_j V(A_1^i,\dots,A_k^i,A_1^j,\dots,A_k^j,C_1,\dots,C_{n-2k}).
$$
\end{proposition}

\begin{proof}
Choose convex bodies $A^i_{j,l}\in \calK^\infty_+(\RR^n)$ as in  Lemma~\ref{lemma:C11}. Put
\begin{displaymath}
  \psi_l =  \sum_{i=1}^N x_i V(A_{1,l}^i,\dots,A_{k,l}^i,\Cdot[n-k]) \in \Val_{n-k}^\infty
\end{displaymath}
and $\phi_l := \pi_\bC \psi_l$. Then the proof of Lemma~\ref{lemma:C11_mixedvolume} shows  $\psi_l\to \phi$ in $\VE_{n-k}$ and hence $\phi_l\to \phi$ in $\VE_{n-k}$.  Observe that
\begin{align*}
  q_\bC(\phi_l-\phi_m,\phi_l-\phi_m)&=(-1)^k L_{\bC_{\setminus0}}(\phi_l-\phi_m) * \pi_\bC(\psi_l-\psi_m)\\
  &= (-1)^k \langle    L_{\bC_{\setminus0}}^\VE (\phi_l-\phi_m),  \pi_\bC^\VF (\psi_l-\psi_m)\rangle.
\end{align*}
Let  $\omega_l= \sum_{i=1}^N x_i \omega_{\bA^i_l}$ and $\omega= \sum_{i=1}^N x_i \omega_{\bA^i}$ be differential forms representing $\psi_l$ and $\phi$. Since $*_1 D\omega_{\bA} =  \frac{(n-k)!}{n!}  \cdot   \Omega_{A_{1}} \wedge \cdots \wedge \Omega_{A_{k}}$, it follows that  $D\omega_l-D\omega_m\to 0$ in $L^\infty$, hence $\psi_l-\psi_m\to 0$ in $\VF_{n-k}$. We conclude that
\begin{displaymath}
 q_\bC(\phi_l-\phi_m,\phi_l-\phi_m)\to 0\quad \text{as }l,m\to \infty.
\end{displaymath}

By Lemma~\ref{lemma:closed_quadratic} and \eqref{eq:qc_explicit}, $\phi$ lies in the domain of $q^\VE_\bC$ and
\begin{align*}
 q^\VE_\bC(\phi,\phi)
 & = \lim_{l \to \infty} q_\bC(\phi_l,\phi_l)\\
  & = (-1)^k  \frac{(n-k)!^2}{n!}\sum_{i,j=1}^N x_ix_j V(A_1^i,\dots,A_k^i,A_1^j,\dots,A_k^j,C_1,\dots,C_{n-2k}).
\end{align*}

\end{proof}

Proposition~\ref{prop:C11} and Theorem~\ref{thm:mainQuadratic}(a) allow us to weaken the hypothesis in Theorem~\ref{thm:mainMV} as follows.

\begin{corollary}
\label{cor:K11}
	The conclusion of Theorem~\ref{thm:mainMV} holds also if the convex bodies $A_j^i$ belong merely to $\mathcal K^{1,1}(\RR^n)$.
\end{corollary}
\begin{proof}
	Let the hypothesis of  Theorem~\ref{thm:mainMV} be satisfied for 	convex bodies $A_j^i$ from $\mathcal K^{1,1}(\RR^n)$.
By Lemma~\ref{lemma:C11_mixedvolume}, the valuation 
\begin{displaymath}
  \phi = \sum_{i=1}^N x_i V(A_1^i,\dots,A_k^i,\cdot[n-k])
\end{displaymath}
is primitive as an element of $\VE_{n-k}$ and thus belongs by Proposition~\ref{prop:C11}  to the domain of $q^\VE_\bC$. Since $q^\VE_\bC$ is strictly positive by  Theorem~\ref{thm:mainQuadratic}(a), Proposition~\ref{prop:C11} yields 
\begin{displaymath}
  (-1)^k\sum_{i,j=1}^N x_ix_j V(A_1^i,\dots,A_k^i,A_1^j,\dots,A_k^j,C_1,\dots,C_{n-2k})\geq 0
\end{displaymath}
with equality if and only if $\phi=0$ as an element of $\VE_{n-k}$. By Lemma~\ref{lemma:C11_mixedvolume} the latter is equivalent to $\sum_{i=1}^N x_i V(A_{1}^i,\dots,A_{k}^i,K[n-k])=0$ for all convex bodies $K$.
 \end{proof}

\bibliographystyle{abbrv}
\bibliography{ref_papers,ref_books}

\end{document}